\documentclass[12pt]{amsart} 
 \normalfont

\usepackage[psamsfonts]{amssymb}
\usepackage{pifont}
\theoremstyle{plain}

\input{xy}
\xyoption{all}

\newtheorem{theorem}{Theorem}[section]
\newtheorem{lemma}[theorem]{Lemma}
\newtheorem{proposition}[theorem]{Proposition}
\newtheorem{corollary}[theorem]{Corollary}

\def\bc{\mathbb{C}}

\def\bz{\mathbb{Z}}
\def\br{\mathbb{R}}

\def\fB{\mathfrak{B}}

\def\fN{\mathfrak{N}}

\def\fT{\mathfrak{T}}

\def\fV{\mathfrak{V}}

\def\fZ{\mathfrak{Z}}
\def\fs{\mathfrak{s}}
\def\fo{\mathfrak{o}}

\newcounter{commentlabel}
\begin{document}
  \title{Central conjugate locus of 2-step nilpotent Lie groups}
\author{Patrick Eberlein}
\date{\today}
\maketitle

\section{Introduction}

\noindent  The goals of this article are twofold : 1)  to compute the conjugate locus of a geodesic that lies in the center of a simply connected, 2-step nilpotent Lie group N with a left invariant metric  2)  compare the isometry types of two such nilpotent Lie groups $N_{1}, N_{2}$ whose conjugate loci for central geodesics are "the same" in a suitable sense.  The first goal is reached in Proposition 4.1 and Corollary 4.5.  The second goal is elusive.  We prove a partial result in Proposition 9.1.
\newline

\noindent We use [E1] as a general reference.  Let $\fN$ denote a 2-step nilpotent Lie algebra ; that is $[[X,Y],Z] = 0$ for all $X,Y,Z \in \fN$ but $\fN$ is not abelian. Let N denote the simply connected nilpotent Lie group with Lie algebra $\fN$.  Let $[\fN,\fN]$ denote the commutator subalgebra of $\fN$.  The Lie algebra $\fN$ is said to be of $\mathit{type~ (p,q)}$ if $[\fN , \fN]$ has dimension p and codimension q in $\fN$.   Let $\fZ$ denote the center of $\fN$ and  note that $[\fN , \fN] \subseteq \fZ$. 
\newline

\noindent  Let $\langle , \rangle$ denote an inner product on $\fN$, and let  $\langle , \rangle$ also denote the corresponding left invariant metric on N.  We write $\fN = \fV \oplus \fZ$, where $\fV = \fZ^{ \perp}$.  The Lie group exponential map $exp : \fN \rightarrow N$ is known to be a diffeomorphism ; see for example [R, p. 6].
\newline

\noindent Let $Z \in \fZ$ be given.  Let $j(Z) : \fV \rightarrow \fV$ be the skew symmetric linear  map given by $\langle j(Z) X , Y \rangle = \langle [X,Y] , Z \rangle$ for all $X,Y \in \fV$.  Let $\gamma_{Z}(t) = exp(tZ)$.  The curve $\gamma_{Z}$  is a geodesic in N, as one can see, for example from Proposition 3.1 of [E1] .  Let   $R_{Z} : \fV \rightarrow \fV$ be the symmetric curvature operator given by $R_{Z}(X) = R(X,Z)Z$.  We recall from Proposition 2.3 of [E1]  that $R_{Z} = - \frac{1}{4}~j(Z)^{2}$.  Since j(Z) is skew symmetric the eigenvalues of $R_{Z}$ are nonnegative for all Z $\in \fZ$.  It follows that $Ric(Z,Z) = trace~R_{Z} \geq 0$, with equality $\Leftrightarrow j(Z) \equiv 0 \Leftrightarrow$ Z is orthogonal to $[\fN , \fN]$.  If $Z \in [\fN , \fN], Z\neq 0$, then $Ric(\gamma_{Z}Õ(t), \gamma_{Z}Õ(t)) = Ric(Z,Z) > 0$ for all $t \in \br$ ; see for example Proposition 2.5 of [E1].  It then follows that $\gamma_{Z}$ must have conjugate points along $\gamma_{Z}$ ; see for example the proof of Theorem 19.4 in [M] or Theorem 1.26 of [CE].   
\newline

\noindent  The discussion above shows that if $[\fN , \fN]$ is a proper subspace of $\fZ$, the center of $\fN$, and if Z is a unit vector in $\fZ$ orthogonal to $[\fN , \fN]$, then $j(Z) \equiv 0$ and $R_{Z} = - \frac{1}{4}~j(Z)^{2} \equiv 0$ on $\fV$.  In fact, $R_{Z} \equiv 0$ on $\fN$ by the formulas in section 2.3 of [E1]. The calculations in section 4 of this article simplify in this case and show that the geodesic $\gamma_{Z}$ has no conjugate points.  Moreover, $\gamma_{Z}$ is tangent to the Euclidean factor of N by Proposition 2.7 of [E1].
\newline

\noindent  In the sequel we shall restrict attention to those Z in $[\fN , \fN]$ for reasons that are apparent from the discussion above.
\newline

\noindent Let $\fs \fo(q,\br)$ denote the real vector space of q x q skew symmetric matrices.  The map $j : [\fN, \fN] \rightarrow \fs \fo(q,\br)$ is injective by the discussion above, and hence $ p \leq dim~\fs \fo(q,\br) = \frac{1}{2}q(q-1)$.
\newline

\noindent In  sections 2 and 3  we compute the Jacobi vector fields on the geodesic $\gamma_{Z}$ in terms of an eigenbasis of $R_{Z}$.  In section 4 we determine the conjugate locus and the multiplicities of $\gamma_{Z}$.  The conjugate locus of an arbitrary geodesic of N was computed in [JLP] in the special case that $j(Z)^{2} = \langle S(Z) , Z \rangle A$ for all $Z \in \fZ$, where $S : \fZ \rightarrow \fZ$ is a positive definite symmetric linear operator, and $A : \fV \rightarrow \fV$ is a fixed negative definite symmetric linear operator, independent of Z.  In this case j(Z) is nonsingular for all nonzero $Z \in \fZ$ and hence q must be even.
\newline 

\noindent We say that $t > 0$ is a $\mathit{conjugate~ value}$ of $\gamma_{Z}$ if $\gamma_{Z}(t)$ is conjugate along $\gamma_{Z}$ to $\gamma_{Z}(0) = e$, the identity of N.  For $Z \in [\fN, \fN]$ and $t > 0$ we define $c(Z,t) = \{\mu \in \frac{2 \pi i}{t}\bz^{+} : \mu~\rm{is~an~eigenvalue~of }~j(Z) \}$. For $\mu \in c(Z,t)$ let $m(\mu)$ denote the multiplicity of $\mu$ as an eigenvalue of j(Z). Let $m(t)$ denote the multiplicity of t as a conjugate value of $\gamma_{Z}$.  
\newline  

\noindent In Corollary 4.5 we show that 1)  $t > 0$ is a conjugate value of $\gamma_{Z} \Leftrightarrow$ c(Z,t) is nonempty and 2)  For a conjugate value $t > 0, m(t) = 2~\sum_{\mu \in c(Z,t)} m(\mu)$.  In particular every conjugate value $t > 0$ has even multiplicity.  We given an alternate description of the conjugate locus of $\gamma_{Z}$, but without multiplicities, in Proposition 4.1 and with multiplicities in Proposition 4.4.
\newline 

\noindent Let $t > 0$ be a conjugate value.  Note that $c(Z,t) \subseteq c(Z,kt)$ for all positive integers k.  Hence $m(t) \leq m(kt)$ and kt is a conjugate value of $\gamma_{Z}$ for all positive integers k.
\newline

\noindent  Corollary 4.5 shows that for nonzero Z in $[\fN , \fN]$ the eigenvalues and their multiplicities of j(Z) determine the conjugate values and their multiplicities of the geodesic $\gamma_{Z}$.  In section 5 we consider the converse assertion.   A conjugate value $t > 0$ for $\gamma_{Z}$ is $\mathit{primitive}$ if t cannot be written as $mt_{0}$ for some integer $m \geq 2$ and some conjugate value $t_{0} > 0$ for $\gamma_{Z}$. In Proposition 5.1 we show that $\gamma_{Z}$ has at most $P = [\frac{1}{2}q]$ distinct primitive conjugate values, and if equality holds then the eigenvalues of j(Z) are distinct and determined by the conjugate locus of $\gamma_{Z}$.
\newline

\noindent  Let $t_{1}, ... t_{m}$ be the primitive conjugate values of $\gamma_{Z}$, where $m \leq P$ is some integer.  The discussion above shows that the set of conjugate values for $\gamma_{Z}$ is $\bigcup_{k=1}^{m} rt_{k}, r \in \bz^{+}$.
\newline

\noindent  In section 6 we review some basic facts about the topology of Grassmann manifolds that will be useful later.
\newline

\noindent  In section 7 we take a first step in determining those nonzero Z in $[\fN , \fN]$ for which $\gamma_{Z}$ has $P = [\frac{1}{2}q]$ distinct primitive conjugate values.  We describe  the space I(p,q) of isometry classes of metric, 2-step nilpotent Lie algebras $\{\fN , \langle , \rangle \}$.  We show that I(p,q) can be identified with $G(p, \fs \fo(q,\br)) / O(q,\br)$, where $G(p, \fs \fo(q,\br))$ denotes the Grassmann manifold of p-dimensional subspaces of $\fs \fo(q,\br)$ and $O(q,\br)$ acts on $G(p, \fs \fo(q,\br)$ by conjugation. The description of I(p,q) is known (cf. [GW]) but proofs seem to be lacking so I include a discussion here.
\newline

\noindent  In section 8 we show that there exists a dense open subset $U'$ in I(p,q) with the following property :  Let $\{\fN, \langle , \rangle \}$ be a metric, 2-step nilpotent Lie algebra whose isometry class lies in $U'$, and let $\{N , \langle, \rangle \}$ denote the corresponding simply connected, 2-step nilpotent Lie group with left invariant metric $\langle , \rangle$.  Then there exists a dense open subset O of full Lebesgue measure in $[\fN , \fN]$ such that the geodesic $\gamma_{Z}$ has $P = [\frac{1}{2}q]$ distinct primitive conjugate values for each nonzero element Z of O.
\newline

\noindent  In section 9 we obtain a partial answer to the question of whether the conjugacy locus of a geodesic $\gamma_{Z}, Z \in [\fN ,\fN]$, determines the eigenvalues and their multiplicities of j(Z).  More precisely, let $U'$ be the dense open subset of I(p,q) that occurs in the statement of Corollary 8.3.  Let $\{\fN_{1}, \langle , \rangle_{1} \}$ and $\{\fN_{2}, \langle , \rangle_{2} \}$ be two metric, 2-step nilpotent Lie algebras of type (p,q) with the following properties :

	1)  The isometry class of $\{\fN_{1}, \langle , \rangle_{1} \}$ lies in $U'$.
	
	2)  There exists a linear isomorphism $\varphi : [\fN_{1} , \fN_{1}] \rightarrow [\fN_{2} , \fN_{2}] $ such that for all $Z \in [\fN_{1} , \fN_{1}]$ the geodesics $\gamma_{Z}$ and $\gamma_{\varphi(Z)}$ have the same conjugate locus in the corresponding simply connected, 2-step nilpotent Lie groups $N_{1}$ and $N_{2}$.
\newline

\noindent Then $j(Z)$ and $j(\varphi(Z))$ have the same eigenvalues and multiplicities for all $Z \in [\fN_{1}, \fN_{1}]$.
\newline

\noindent  It is natural to ask if conditions 1) and 2) in the result above imply that there exists a linear isometry $\varphi : \{\fN_{1}, \langle , \rangle_{1} \} \rightarrow \{\fN_{2}, \langle , \rangle_{2} \}$ that is also a Lie algebra isomorphism.  This is known to be false if $p =2$ (cf. Theorem 2.2 of [GW]), but an affirmative answer might be possible if p is sufficiently large ; the condition 2) above becomes more stringent as p increases.
\newline

\noindent  In section 10 we include some material on polynomial maps and Zariski open sets that was needed for the proof of the result from section 9 stated above.  This material might not be familiar to a reader with a traditional background in differential geometry.
\newline

\noindent  $\mathbf{Acknowledgments}$ The author is grateful for many conversations with Ernst Heintze that showed the question above has a positive answer for $p = 2$ and $q = 3$ or $q = 4$ but suggested that the question has a negative answer for $p =2$ and q large.  The author is grateful to Carolyn Gordon for providing the reference (Theorem 2.2 of  [GW]) that confirmed a negative answer for $p =2$ and $q \neq 3,4,6$.
	
\section{Jacobi vector fields on a central geodesic}

\noindent Let $Z \in [\fN ,\fN]$, and let $  \{ i \lambda_{1}  - i \lambda_{1}, ... ,  i  \lambda_{M}, - i \lambda_{M}  \}$ be the nonzero eigenvalues of j(Z), possibly not all distinct, where $\lambda_{i} > 0$ for every i.  It follows that $\{ \frac{1}{4} \lambda_{1}^{2}, ... , \frac{1}{4} \lambda_{M}^{2} \}$ are the eigenvalues of $R_{Z}$, possibly not all distinct.   Choose a basis $\fB = 
 \{A_{1}, ... , A_{r}, B_{1}, ... , B_{M}, j(Z) B_{1}, ... , j(Z) B_{M}, Z_{1}, ... , Z_{s} \}$ of $\fN$ with the following properties :
	
	a)  $Ker~j(Z)  = Ker~j(Z)^{2} = Ker~R_{Z} \cap \fV = \br-span\{A_{1}, ... , A_{r} \}$
	
	b)  $\fZ = \br -  span \{Z_{1}, ... , Z_{s} \}$
	
	c) $R_{Z}(B_{i}) = \frac{1}{4} \lambda_{i}^{2}~B_{i}$ for $1 \leq i \leq M$.
\newline

\noindent Note that $R_{Z}~ (j(Z) (B_{i})) = \frac{1}{4} \lambda_{i}^{2} ~j(Z)(B_{i})$ for $1 \leq i \leq M$ since $R_{Z}$ commutes with j(Z).  Moreover, $R_{Z} \equiv 0$ on $\fZ$ by section 2.3 of [E1]. Hence the basis $\fB$ consists of eigenvectors of $R_{Z}$.  If j(Z) is nonsingular, then $r = 0$.  If j(Z) $\equiv 0$ (i.e. Z is orthogonal to $[\fN , \fN]$), then $R_{Z} \equiv 0$ and $M = 0$.
\newline

\begin{proposition} Let $\gamma(t) = \gamma_{Z}(t)$.  Let $Y(t) =  \sum_{i=1}^{r} a_{i}(t) A_{i} (\gamma(t)) + \newline  \sum_{j=1}^{M} b_{j}(t) B_{j}(\gamma(t)) +  \sum_{j=1}^{M} c_{j}(t) (j(Z) (B_{j})(\gamma(t))  +  \sum_{\alpha = 1}^{s} d_{\alpha}(t) Z_{\alpha}(\gamma(t))$ be a vector field on the geodesic $\gamma$.  Then Y(t) is a Jacobi vector field $\Leftrightarrow$

	1)  $a_{i}(t) = c_{i} + d_{i} t$ \hspace{.5in}  for some  real numbers $c_{i},d_{i}$ and for $1 \leq i \leq r$
	
	2)  $d_{\alpha}(t) = e_{\alpha} + f_{\alpha} t$ \hspace{.4in} for some  real numbers $e_{\alpha},f_{\alpha}$ and for $1 \leq \alpha \leq s$.
	
	3)  There exist constants $\alpha_{j}, \beta_{j}, \gamma_{j}, \delta_{j}, 1 \leq i \leq M$ such that
	
	\hspace{.3in} a) $b_{j}(t) = \alpha_{j}~ cos( \lambda_{j}~t) + \beta_{j}~ sin(\lambda_{j}~t) + \gamma_{j}$.
	
	\hspace{.3in} b) $c_{j}(t) = - \frac{\beta_{j}} {\lambda_{j}}~ cos(\lambda_{j}~t) +  \frac{\alpha_{j}}{\lambda_{j}}~ sin(\lambda_{j}~t) + \delta_{j}$.
\end{proposition}

\begin{proof}  We recall from Proposition 2.2 of [E1] that $\nabla_{\gamma'(t)} X = \nabla_{Z}X(\gamma(t)) = \newline - \frac{1}{2}~ j(Z)(X)(\gamma(t))$ for all X $\in \fV$ and  $\nabla_{\gamma'(t)} Z' = \nabla_{Z}Z'(\gamma(t)) = 0$ for all $Z' \in \fZ$.  Furthermore $\nabla_{\gamma'(t)}~A_{i} = (\nabla_{Z}A_{i})(\gamma(t)) = - \frac{1}{2}~(j(Z)A_{i})(\gamma(t)) = 0$ for $1 \leq i \leq r$.  Note that $R_{Z}A_{i} = - \frac{1}{4} j(Z)^{2}A_{i} = 0$ for $1 \leq i \leq r$  Moreover, $- \frac{1}{2} j(Z)^{2}(B_{i}) = 2~R_{Z}(B_{i}) = \frac{1}{2} \lambda_{i}^{2} B_{i}$ for $1 \leq i \leq M$.  Routine computations now yield

\hspace{.5in} (1)  $Y'(t) = \sum_{i=1}^{r} a_{i}'(t) A_{i}(\gamma(t)) + \sum_{j=1}^{M} \{b_{j}'(t) + \frac{1}{2} \lambda_{j}^{2}~c_{j}(t)\}~ B_{j}(\gamma(t))$ \newline $\sum_{j=1}^{M}  \{ - \frac{1}{2} b_{j}(t) + c_{j}'(t) \}~(j(Z)(B_{j}))(\gamma(t)) + \sum_{\alpha = 1}^{s} d_{\alpha}'(t) Z_{\alpha}(\gamma(t))$

\noindent Similarly we obtain

\hspace{.5in} (2)  $Y''(t) = \sum_{i=1}^{r} a_{i}''(t) A_{i}(\gamma(t)) + \sum_{j=1}^{M} \{b_{j}''(t) +  \lambda_{j}^{2}~c_{j}'(t)~ - \frac{1}{4}\lambda_{j}^{2}~b_{j}(t) \} ~B_{j}(\gamma(t))$ \newline + $\sum_{j=1}^{M}  \{ - b_{j}'(t) + c_{j}''(t) - \frac{1}{4}\lambda_{j}^{2} c_{j}(t) \}~(j(Z)(B_{j}))(\gamma(t)) + \sum_{\alpha = 1}^{s} d_{\alpha}''(t) Z_{\alpha}(\gamma(t))$

\noindent We compute

\hspace{.5in} (3) $(R_{Z} Y)(\gamma (t) = \sum_{j=1}^{M}  \frac{1}{4}\lambda_{j}^{2}~ b_{j}(t) B_{j}(\gamma(t)) +  \sum_{j=1}^{M}  \frac{1}{4}\lambda_{j}^{2}~c_{j}(t) (j(Z) (B_{j})(\gamma(t)))$

\noindent From (2) and (3) we obtain

 \hspace{.5in} (4) $Y''(t) + (R_{Z}Y)(\gamma(t)) = \sum_{i=1}^{r} a_{i}''(t) A_{i}(\gamma(t)) + \sum_{j=1}^{M} \{b_{j}''(t) +  \lambda_{j}^{2}~c_{j}'(t)\}  ~B_{j}(\gamma(t))$ \newline + $\sum_{j=1}^{M}  \{ - b_{j}'(t) + c_{j}''(t) \}~(j(Z)(B_{j}))(\gamma(t)) + \sum_{\alpha = 1}^{s} d_{\alpha}''(t) Z_{\alpha}(\gamma(t))$

 \noindent  From the Jacobi equation $0 = Y''(t) + (R_{Z}Y)(\gamma(t))$ we obtain

\hspace{.5in} (5)

\hspace{.7in} (a)  $a_{i}''(t) \equiv 0 \hspace{1.5in}$ for $1 \leq i \leq r$

 \hspace{.7in} (b) $d_{\alpha}''(t) \equiv 0 \hspace{1.5in}$ for $1 \leq \alpha \leq s$

 \hspace{.7in} (c) $b_{j}''(t) +  \lambda_{j}^{2} c_{j}'(t) \equiv 0$  \hspace{.7in} for $1 \leq j \leq M$

  \hspace{.7in} (d) $c_{j}''(t) - b_{j}'(t) \equiv 0$  \hspace{.9in} for $1 \leq j \leq M$

 \noindent  The assertion of the Proposition now follows from routine calculations.
 \newline
\end{proof}

\section{A~ natural~ basis~ for~ the~ Jacobi~ vector~ fields~ on~ $\gamma_{Z}$}

\begin{proposition}  The following vector fields Y(t) form a basis of the Jacobi vector fields on $\gamma_{Z}$, and their first covariant derivatives obey the given relations :

1)  $Y(t) = A_{i}(\gamma(t)), \hspace{.3in} Y'(t) \equiv 0 \hspace{.97in} 1 \leq i \leq r$
\newline

2)  $Y(t) = t~A_{i}(\gamma(t)), \hspace{.2in} Y'(t) = A_{i}(\gamma(t)) \hspace{.5in} 1 \leq i \leq r$
\newline

3)  $Y(t) = Z_{j}(\gamma(t)), \hspace{.3in} Y'(t) \equiv 0 \hspace{.97in} 1 \leq j \leq s$
\newline 

4)  $Y(t) = t~Z_{j}(\gamma(t)), \hspace{.2in} Y'(t) \equiv Z_{j}(\gamma(t)) \hspace{.5in} 1 \leq j \leq s$
\newline

5)  $Y(t) = V_{j}(t) = B_{j}(\gamma(t)) \hspace{.2in} ; \hspace{.2in} Y(t) = W_{j}(t) = j(Z) B_{j}(\gamma(t)) \hspace{.2in} 1 \leq j \leq M$

    \hspace{.15in} $V_{j}'(t) = - \frac{1}{2} W_{j}(t) \hspace{.2in}; \hspace{.2in} W_{j}'(t) = \frac{1}{2} \lambda_{j}^{2} V_{j}(t)\hspace{.2in} 1 \leq j \leq M$ 
\newline
    
 6)  $Y(t) = X_{j}(t) = cos(\lambda_{j}t)~ B_{j}(\gamma(t)) + \frac{1}{\lambda_{j}}~ sin(\lambda_{j}t)~j(Z) B_{j}(\gamma(t))$
 
\hspace{.15in} $Y(t) = Y_{j}(t) = sin(\lambda_{j}t)~ B_{j}(\gamma(t)) - \frac{1}{\lambda_{j}}~ cos(\lambda_{j}t)~j(Z) B_{j}(\gamma(t))$

\hspace{.15in} $X_{j}'(t) = - \frac{1}{2} \lambda_{j}~Y_{j}(t) \hspace{.2in} ;  \hspace{.2in} Y_{j}'(t) =  \frac{1}{2} \lambda_{j}~X_{j}(t) \hspace{.2in} 1 \leq j \leq M$
\end{proposition}

\begin{proof}  The vector fields in 1) through 4) above are clearly Jacobi vector fields on $\gamma_{Z}$ by the previous result.  The vector fields $V_{j},W_{j}$ in 5) correspond to the vector fields in 3) of Proposition 2.1 described by $\gamma_{j} = 1, \alpha_{j} = \beta_{j} = \delta_{j} = 0$ and $\delta_{j} = 1, \alpha_{j} = \beta_{j} = \gamma_{j} = 0$ respectively.  The vector fields $X_{j},Y_{j}$ in 6) correspond to the vector fields in 3) of Proposition 1.1 described by $\alpha_{j} = 1, \beta_{j} = \gamma_{j} = \delta_{j} = 0$ and $\beta_{j} = 1, \alpha_{j} = \gamma_{j} = \delta_{j} = 0$ respectively.
\newline  

 \noindent If the vector field Y(t) in Proposition 2.1 is a Jacobi vector field, then it is routine to show that $Y(t) = \sum_{j=1}^{M} \{\alpha_{j} X_{j}(t) + \beta_{j} Y_{j}(t) + \gamma_{j} V_{j}(t) + \delta_{j} W_{j}(t) \} + \sum_{i=1}^{r} (c_{i} + d_{i} t) A_{i}(\gamma(t)) + \sum_{\alpha = 1}^{s} (e_{\alpha} + d_{\alpha} t) Z_{\alpha}(\gamma(t))$.  Hence the vector fields listed above in 1) through 6) span the vector space $J(\gamma_{Z})$ of Jacobi vector fields on $\gamma_{Z}$.  These vector fields are linearly independent since the number of them is $2 \cdot$ dimension $\fN$, which is also the dimension of $J(\gamma_{Z})$.
 \end{proof}

\section{Conjugate~points~ on~ $\gamma_{Z}$}

\begin{proposition}  Let Z be a nonzero element of $[\fN , \fN]$.  The point $\gamma_{Z}(t)$ is conjugate to $\gamma_{Z}(0) = e \Leftrightarrow t =  \frac{2 m \pi}{\lambda}$ for some positive integer m and some positive number $\lambda$ such that $i \lambda$ is an eigenvalue of j(Z).
\end{proposition}

\noindent  Before proving this result we point out an immediate consequence

\begin{corollary}  Let Z be a nonzero element of $[\fN , \fN]$.  Let $\gamma_{Z}(t)$ be conjugate to $\gamma_{Z}(0) = e$ for some $t > 0$.  Then $\gamma_{Z}(mt)$ is conjugate to $\gamma_{Z}(0)$ for every positive integer m.
\end{corollary}

\noindent  We now begin the proof of the Proposition.  Let $\{i \lambda_{1}, - i \lambda_{1}, i \lambda_{2}, - i \lambda_{2} , ... , i \lambda_{M}, - i \lambda_{M} \}$ be the nonzero eigenvalues of j(Z), where $\lambda_{i} > 0$ for $1 \leq i \leq M$.  For $1 \leq j \leq M$ let $\xi_{j}(t) = X_{j}(t) - V_{j}(t)$ and $\eta_{j}(t) = Y_{j}(t) + \frac{1}{\lambda_{j}} W_{j}(t)$, where $X_{j}, Y_{j}, V_{j}, W_{j}$ are defined in Proposition 3.1.   Let $t =  \frac{2 m \pi}{\lambda}$ for some nonzero integer m and some positive number $\lambda$ such that $i \lambda$ is an eigenvalue of j(Z).  Then $\lambda = \lambda_{j}$ for some $1 \leq j \leq M$. It follows that  $\xi_{j}(t) = \eta_{j}(t) = 0$.   Hence $\gamma_{Z}(0)$ is conjugate to $\gamma_{Z}(t)$.  To prove the reverse assertion of the Proposition requires a bit more work.

\begin{lemma}  Let $\gamma$ denote $\gamma_{Z}$ and let $J_{0}(\gamma)$ denote the vector space of Jacobi vector fields on $\gamma$ with $\gamma(0) = 0$.  Then \newline $\fB = \{tA_{1}, ... , tA_{r}, tZ_{1}, ... , tZ_{s}, \xi_{1}, ... , \xi_{M}, \eta_{1}, ... , \eta_{M} \}$ is a basis for $J_{0}(\gamma)$.
\end{lemma}

\begin{proof}  By inspection $\fB \subset J_{0}(\gamma)$.  Note that $|\fB| = r + s + 2M = dim~\fN = dim~J_{0}(\gamma)$.  To prove the linear independence of $\fB$ it suffices to show that $\fB$ spans $J_{0}(\gamma)$.
\newline

\noindent Let $Y(t) \in J_{0}(\gamma)$.  Write 

\hspace{.5in}  $(*) Y(t) =  \sum_{i=1}^{r} a_{i}(t) A_{i} (\gamma(t)) + \sum_{j=1}^{M} b_{j}(t) B_{j}(\gamma(t)) +   \newline \sum_{j=1}^{M} c_{j}(t) (j(Z) (B_{j})(\gamma(t)))  + \sum_{\alpha = 1}^{s} d_{\alpha}(t) Z_{\alpha}(\gamma(t))$ 

\noindent where the functions $a_{i}, b_{j}, c_{j}$ and $d_{\alpha}$ satisfy the conditions of Proposition 2.1.  Since $Y(0) = 0$ the functions $a_{i},b_{j}, c_{j}$ and $d_{\alpha}$ all vanish at $t = 0$. From Proposition 2.1 it follows that there exist real numbers $x_{1}, ... , x_{r}, d_{1}, ... , d_{p}, \alpha_{1}, ... , \alpha_{M}, \beta_{1}, ... , \beta_{M}$ such that 

\hspace{1in}  $a_{i}(t) = x_{i}t$, \hspace{.1in} $1 \leq i \leq r$.  

\hspace{1in}  $d_{\alpha}(t) = d_{\alpha}t$, \hspace{.1in} $1 \leq\alpha \leq s$. 

\hspace{1in}  $b_{j}(t) = \alpha_{j}~cos(\lambda_{j}t) + \beta_{j}~sin(\lambda_{j} t) - \alpha_{j}$ \hspace{.1in} $1 \leq j \leq M$. 

\hspace{1in}  $c_{j}(t) = - \frac{\beta_{j}}{\lambda_{j}}~cos(\lambda_{j}t) +  \frac{\alpha_{j}}{\lambda_{j}}~sin(\lambda_{j} t) + \frac{\beta_{j}}{\lambda_{j}}$ \hspace{.1in} $1 \leq j \leq M$.

\noindent  Plugging this  information into $(^{*})$ yields

\noindent $ (**)~Y(t) = \sum_{i=1}^{r} x_{i} (tA_{i})(\gamma(t)) + \sum_{\alpha=1}^{s} d_{\alpha} (t Z_{\alpha})(\gamma(t)) + \sum_{j=1}^{M} \alpha_{j} \xi_{j}(t) + \sum_{j=1}^{M} \beta_{j} \eta_{j}(t)$.  It follows that $\fB = \{tA_{1}, ... , tA_{r}, tZ_{1}, ... , tZ_{s}, \xi_{1}, ... , \xi_{M}, \eta_{1}, ... , \eta_{M} \}$  spans $J_{0}(\gamma)$.  Hence $\fB$ is a basis for $J_{0}(\gamma)$.
\newline

\noindent  We now complete the proof of the Proposition.  Let Y(t) be a Jacobi vector field on $\gamma(t)$, not identically zero, such that $Y(0) = 0$ and $Y(t_{0}) = 0$ for some positive number $t_{0}$. The condition $Y(t_{0}) = 0$, where $t_{0} > 0$ implies that the vector fields $tA_{i}$ and $tZ_{\alpha}$ have zero components in the expression for $Y(t)$ in $(**)$.  Hence  $Y(t) = \sum_{j=1}^{M} \alpha_{j} \xi_{j}(t) + \sum_{j=1}^{M} \beta_{j} \eta_{j}(t)$.  The conditions $b_{j}(t_{0}) = c_{j}(t_{0})  = 0$ for $1 \leq j \leq M$ imply that  $0 = (\alpha_{j}^{2} + \beta_{j}^{2})~sin(\lambda_{j} t)$ and $(\alpha_{j}^{2} + \beta_{j}^{2}) = (\alpha_{j}^{2} + \beta_{j}^{2})~cos~(\lambda_{j} t)$ for $1 \leq j \leq M$. Hence if $\alpha_{j}^{2} + \beta_{j}^{2} > 0$, then $t_{0} = \frac{2 m_{j} \pi}{\lambda_{j}}$ for some $1 \leq j \leq M$ and some nonzero integer $m_{j}$.  Since $Y(t) = \sum_{j=1}^{M} \alpha_{j} \xi_{j}(t) + \sum_{j=1}^{M} \beta_{j} \eta_{j}(t)$ is not identically zero, it follows that $\alpha_{j}^{2} + \beta_{j}^{2} > 0$ for some j. 
\newline

\noindent  The argument above shows that if $\gamma_{Z}(0)$ and $\gamma_{Z}(t_{0})$ are conjugate for some $t_{0} > 0$, then $t_{0} = \frac{2 m_{j} \pi}{\lambda_{j}}$ for some $1 \leq j \leq M$ and some positive integer $m_{j}$.  This completes the proof of the Proposition.
\end{proof}

\noindent $\mathbf{Remark} $ From the proof of the previous result it follows that if  $\alpha_{i}^{2} + \beta_{i}^{2} > 0$ and  $\alpha_{j}^{2} + \beta_{j}^{2} > 0$ for distinct integers i,j, then $\frac{\lambda_{i}}{\lambda_{j}}$ is rational.

\begin{proposition}  Let  a nonzero element $Z \in [\fN , \fN]$ be given, and let $\gamma$ denote  $\gamma_{Z}$.  For $t_{0} > 0$ let $J_{t_{0}}(\gamma) = \{Y \in J_{0}(\gamma) : Y(t_{0}) = 0\}$.  Let $\{\lambda_{1}, ... , \lambda_{M} \}$  be those positive numbers such that $\{\pm i \lambda_{1} , ... , \pm i \lambda_{M}\}$  are the nonzero eigenvalues of j(Z).  Let $I = \{j : 1 \leq j \leq M~\rm{and}~ t_{0} = \frac{2 \pi m_{j}}{\lambda_{j}}~\rm{for~some~positive~integer}~m_{j} \}$. Then $J_{t_{0}}(\gamma) = \br$-span $ \{\xi_{j}, \eta_{j} : j \in I \}$.  In particular the conjugate point $\gamma(t_{0})$ has even multiplicity.
\end{proposition}

\begin{proof}  If $j \in I$, then $\xi_{j}(t_{0}) = \eta_{j}(t_{0}) = 0$.  Hence $J_{t_{0}}(\gamma)  \supset \br~span - \{\xi_{j}, \eta_{j} : j \in I \}$.  Conversely suppose $Y(t) \in J_{t_{0}}(\gamma)$.  By the proof of Proposition 4.1 there exist real numbers $\alpha_{k} , \beta_{k} : 1 \leq k \leq M$ such that $Y(t) = \sum_{k=1}^{M} \alpha_{k}~ \xi_{k}(t) +   \sum_{k=1}^{M} \beta_{k}~ \eta_{k}(t)$.  Let $I_{1} = \{j : 1 \leq j \leq M~\rm{and}~ \alpha_{j} \neq 0\}$. Let $I_{2} = \{k : 1 \leq k \leq M~\rm{and}~ \beta_{k} \neq 0\}$.  Clearly, $Y(t) = \sum_{j \in~ I_{1}} \alpha_{j}~ \xi_{j}(t) + \sum_{k \in~ I_{2}} \beta_{k}~ \eta_{k}(t)$.  The proof of Proposition 4.1 shows that if $j \in I_{1}$ or $j \in I_{2}$, then $t_{0} = \frac{2 \pi m_{j}}{\lambda_{j}}$ for some positive integer $m_{j}$.   Hence $I_{1} \cup I_{2} \subset I$.  It follows that $Y(t) \in  \br~span - \{\xi_{j}, \eta_{j} : j \in I \}$, which completes the proof.
\end{proof}

\begin{corollary}  For $Z \in [\fN , \fN], Z \neq 0$ and $t > 0$ define $c(Z,t) = \{\mu \in \frac{2 \pi i}{t}\bz^{+} : \mu~\rm{is~an~eigenvalue~of }~j(Z) \}$.  For $\mu \in c(Z,t)$ let $m(\mu)$ denote the multiplicity of $\mu$ as an eigenvalue of j(Z). Then 

	1) $t > 0$ is a conjugate value of $\gamma_{Z} \Leftrightarrow c(Z,t)$ is nonempty.
	
	2)  If m(t) is the mutliplicity of $t > 0$ as a conjugate value, then $m(t) = \newline 2~\sum_{\mu \in c(Z,t)} m(\mu)$.  In particular m(t) is even.
\end{corollary}

\begin{proof}  Assertion 1) follows directly from Proposition 4.1 by a straightforward argument.  We prove 2).  Straightforward arguments show that $\mu \in c(Z,t) \Leftrightarrow \mu = i \lambda_{j}$ for some $j \in I$, where I is defined in the statement of Proposition 4.4.  Moreover, by inspection, for $\mu \in c(Z,t), m(\mu)$ is the number of j in I such that $\mu = i \lambda_{j}$.  Hence $\sum_{\mu \in c(Z,t)} m(\mu) = |I|$, and assertion 2) now follows from Proposition 4.4. 
\end{proof}

\noindent $\mathbf{Remark}$  Let $t > 0$ be a conjugate value.  By inspection $c(Z,t) \subseteq c(Z, kt)$ for all positive integers k.  Hence $m(t) \leq m(tk)$ and kt is a conjugate value for all positive integers k.
\newline

\section{Comparison of the conjugate locus of $\gamma_{Z}$ and the eigenvalues of j(Z)}

\noindent  Let $0 \neq Z \in [\fN ,\fN]$ be given, and let $\gamma_{Z}(t) = exp(tZ)$ be the corresponding central geodesic of N.  If $t_{0}$ is a conjugate value of $\gamma_{Z}$, then $mt_{0}$ is a conjugate value of $\gamma_{Z}$ for every positive integer m by Corollary 4.2 or Corollary 4.5.  We say a conjugate value $t_{0}$ is $\mathit{primitive}$ if $t_{0}$ cannot be written as $mt_{1}$ for an integer $m \geq 2$ and a conjugate value $t_{1}$.
\newline

\noindent  Let $0 \neq Z \in [\fN , \fN]$ be given. It follows from Proposition 4.1 and Corollary 4.5 that the eigenvalues and multiplicities  of j(Z) determine the conjugate locus and multiplicity of the geodesic $\gamma_{Z}$.  However, it is not clear if the conjugate locus and multiplicities of $\gamma_{Z}$ always determine the eigenvalues and multiplicities  of j(Z).  We can show that this happens if the number of primitive conjugate values of $\gamma_{Z}$ is maximal.  Roughly speaking, this maximality condition is satisfied  for a generic element Z in a generic metric 2-step nilpotent Lie algebra $\{\fN, \langle , \rangle \}$.  See Proposition 8.2 or Corollary 8.3 for a precise statement.
\newline

\begin{proposition}  Let $q \geq 2$ be an integer, and let $P = [\frac{1}{2} q]$. Let $\{\fN, \langle , \rangle \}$ be a metric Lie algebra of type (p,q), and let N be the simply connected, 2-step nilpotent Lie group with left invariant metric whose metric Lie algebra is $\{\fN, \langle , \rangle \}$.  Let $Z \in [\fN , \fN]$ be nonzero.  Then

	1)  The geodesic $\gamma_{Z}$ of N has at most P primitive conjugate values.
	
	2)  The geodesic $\gamma_{Z}$ of N has P primitive conjugate values $\Leftrightarrow$ the eigenvalues of j(Z) are all distinct, and the ratio of any two distinct nonzero eigenvalues is not an integer of absolute value $\geq 2$.  
	
	3)  If $\{t_{1}, ... , t_{P} \}$ are distinct primitive conjugate values, then $t_{k} = \frac{2 \pi}{\lambda_{k}}$ for $1 \leq k \leq P$, where $\lambda_{k}$ is a positive number such that $i \lambda_{k}$ is an eigenvalue of j(Z).  The nonzero eigenvalues of j(Z) are $\{\pm i \lambda_{1} , ... , \pm i \lambda_{P} \}$.
\end{proposition}

\noindent $\mathbf{Remark}$  With regard to assertion 2) above note that if $i \lambda$ is an eigenvalue of j(Z), then $- i \lambda$ is also an eigenvalue of j(Z), and the ratio of these two distinct eigenvalues is $-1$.  If q is odd, then in 3) zero is an eigenvalue of j(Z) of multiplicity one.
\newline

\noindent As an immediate corollary of assertion 3) of this result we obtain

\begin{corollary} Let $q \geq 2$ be an integer, and let $P = [\frac{1}{2} q]$.  Let $\{\fN, \langle , \rangle \}$ be a metric Lie algebra of type (p,q), and let N be the simply connected, 2-step nilpotent Lie group with left invariant metric whose metric Lie algebra is $\{\fN, \langle , \rangle \}$. Let $Z \in [\fN , \fN]$ be nonzero.  Let the geodesic $\gamma_{Z}$ of N have P distinct primitive conjugate values.  Then the eigenvalues of j(Z) are all distinct and are determined by the conjugate locus and multiplicity of $\gamma_{Z}$.
\end{corollary}

\noindent  We begin the proof of Proposition 5.1.  To prove 1) the next result is useful.

\begin{lemma}  If $t > 0$ is a primitive conjugate value of $\gamma_{Z}$, then $t = \frac{2 \pi}{\lambda}$, where $\lambda > 0$ and $i \lambda$ is an eigenvalue of j(Z).
\end{lemma}

\begin{proof}   Let $t > 0$ be a primitive conjugate value of $\gamma_{Z}$.  By Proposition 4.1 we may write $t = \frac{2 \pi m}{\lambda}$, where m is a positive integer and $i \lambda$ is a nonzero eigenvalue of j(Z).  If $t_{0} = \frac{2 \pi}{\lambda}$, then $t_{0}$ is a conjugate value of $\gamma_{Z}$ by Proposition 4.1.  Hence $t = m t_{0}$, which proves that $m = 1$ since t is a primitive conjugate value.
\end{proof}

\noindent  We now prove 1) of Proposition 5.1.  Suppose that $\{t_{1}, ... , t_{m} \}$ are distinct primitive conjugate values of $\gamma_{Z}$ for some integer $m > P$.  By Lemma 5.3 there exist distinct positive numbers $\{\lambda_{1}, ... , \lambda_{m} \}$ such that $t_{k} = \frac{2 \pi}{\lambda_{k}}$ and $i \lambda_{k}$ is an eigenvalue of j(Z) for $1 \leq k \leq m$.  It follows that j(Z) has $2m \geq 2P + 2 > q$ distinct eigenvalues, which is impossible.
\newline

\noindent  We prove 2).  Suppose first that $\gamma_{Z}$ has P distinct primitive conjugate values $\{t_{1}, ... , t_{P} \}$.  We show that the eigenvalues of j(Z) are all distinct.  By Lemma 5.3 we may write $t_{k} = \frac{2 \pi}{\lambda_{k}}$, where $\lambda_{k} > 0$ and $i \lambda_{k}$ is an eigenvalue of j(Z) for $1 \leq k \leq N$.  If q is even, then $q = 2P$ and the eigenvalues $\{\pm i \lambda_{1}, ... , \pm i \lambda_{P} \}$ are all distinct.  By inspection these are all of the eigenvalues of j(Z).  If q is odd, then $q = 2P +1$ and zero is an eigenvalue of j(Z).  The eigenvalues $\{\pm \lambda_{1}, É , \pm \lambda_{P} \}$ are distinct and nonzero, which proves in this case that all eigenvalues of j(Z) are distinct.
\newline

\noindent We show next that the ratio of nonzero eigenvalues of j(Z) cannot be an integer of absolute value $\geq 2$.  Suppose that $\frac{i \lambda_{j}}{i \lambda_{k}} = \frac{\lambda_{j}}{\lambda_{k}} = m$, where $m \in \bz$ and $|m| \geq 2$.  Then $\frac{t_{k}}{t_{j}} = \frac{\lambda_{j}}{\lambda_{k}} = m$, where $|m| \geq 2$, which contradicts the fact that $t_{k}$ is a primitive conjugate value of $\gamma_{Z}$.
\newline

\noindent  Conversely, suppose that the eigenvalues of j(Z) are all distinct and that the ratio of any two nonzero eigenvalues is not an integer of absolute value $\geq 2$.  We consider first the case that q is even.  Then $q = 2P$ and all eigenvalues of j(Z) are nonzero since the multiplicity of zero as an eigenvalue in this case is always even.  Hence j(Z) has eigenvalues $\{\pm i \lambda_{1}, ... , \pm i \lambda_{P} \}$, where $\{\lambda_{1}, ... , \lambda_{P} \}$ are distinct positive numbers.  Let $t_{k} = \frac{2 \pi}{\lambda_{k}}$ for $1 \leq k \leq P$.  The numbers $\{t_{1}, ... , t_{P} \}$ are distinct conjugate values of $\gamma_{Z}$ by Proposition 4.1.  We show that they are all primitive.  Suppose that $t_{k} = mt$ for some $1 \leq k \leq P$, some positive integer m and some conjugate value t.  By Proposition 4.1 we can write $ t= \frac{2 \pi r}{\lambda_{j}}$ for some $r \in \bz^{+}$ and some $\lambda_{j} , 1 \leq j \leq P$.  Hence $\frac{2 \pi}{\lambda_{k}} = rt = \frac{2 \pi m r}{\lambda_{j}}$, which implies that $\lambda_{j} = m r \lambda_{k}$.  By the hypothesis on the ratios of nonzero eigenvalues of j(Z) it follows that  $ m = r =1$.  Hence  $t_{k} = \frac{2 \pi}{\lambda_{k}}$ is primitive for $1 \leq k \leq P$.
\newline

\noindent  Finally, we consider the case that q is odd ; that is, $q = 2P + 1$.  Since the eigenvalues of j(Z) are assumed to be distinct there exist distinct positive numbers $\{\lambda_{1}, ... , \lambda_{P} \}$ such that the eigenvalues of j(Z) are $\{0, \pm i \lambda_{1}, É , \pm i \lambda_{P} \}$.  The argument above also shows in this case that $t_{k} = \frac{2 \pi}{\lambda_{k}}$ is a primitive conjugate value for $1 \leq k \leq P$.  This completes the proof of 2).
\newline

\noindent 3).  The proof of this assertion is contained in the proof of 2).

\section{Topology of Grassmann manifolds}

\noindent  It will be useful in what follows to review briefly the canonical topology of the Grassmann manifold of p-planes in $\br^{n}$, where $1 \leq p \leq n-1$.  See [MS], pp. 55-59, for details.
\newline

\noindent  Let $V_{0}(p,n)$ denote the set of all orthonormal p-frames $\{f_{1}, ... , f_{p} \}$ in $\br^{n}$, and define a surjective map $q_{0} : V_{0}(p,n) \rightarrow G(p,n)$ by $q_{0}(\{f_{1}, ... , f_{p} \}) = \br$-span$\{f_{1}, ... , f_{p} \}$.  Note that $V_{0}(p,n)$ is a closed bounded subset of $(\br^{n})^{p} = \br^{n} \times ... \times \br^{n}$ (p times).  Hence $V_{0}$ is compact in the subspace topology of $(\br^{n})^{p}$.  If $\fT$ denotes the quotient topology on G(p,n) defined by $q_{0}$, then G(p,n) is compact in this topology.  The topology $\fT$ is Hausdorff.
\newline

\noindent  Let d denote the metric on $(\br^{n})^{p}$ given by $d(f,f')^{2} = \sum_{i=1}^{p} |f_{i} - f_{i}'|^{2}$, where $f = (f_{1}, ... , f_{p})$ and $f '= (f_{1}', ... , f_{p}')$.  The topology of $(\br^{n})^{p}$, and hence by restriction to $V_{0}(p,n)$, is induced by d.
\newline

\noindent Alternatively, let V(p,n) denote the set of all linearly independent sets $\{v_{1}, ... , v_{p} \}$ with p elements in $\br^{n}$.  Let V(p,n) have the subspace topology of $(\br^{n})^{p}$ and define a surjective map $q : V(p,n) \rightarrow G(p,n)$ by $q(\{v_{1}, ... , v_{p} \}) = \br$-span$\{v_{1}, ... , v_{p} \}$.  One may show that the quotient topology on G(p,n) defined by q equals $\fT$.
\newline

\begin{lemma} 

	1)  The map $q_{0} : V_{0}(p,n)  \rightarrow G(p,n)$ is an open map.

	2)  Let $W_{0} \in G(p,n)$ be given, and let H be the subgroup of $O(n,\br)$ that fixes $W_{0}$.  Then G(p,n) is homeomorphic to the coset space $O(n,\br) / H$ equipped with the quotient topology defined by the projection $p : O(n,\br) \rightarrow O(n,\br) / H$.
\end{lemma}

\begin{proof}  Fix $f_{0} = (f_{1}, .. , f_{p}) \in V_{0}(p,n)$ and define $\pi_{0} : O(n,\br) \rightarrow V_{0}(p,n)$ by $\pi_{0}(\varphi) = \varphi(f_{0})$.  The map $\pi_{0}$ is a continuous bijection and hence is a homeomorphism since $O(n,\br)$ is compact and $V_{0}(p,n)$ is Hausdorff.
\newline

\noindent Let $W_{0} = q_{0}(f_{0})$ and define a surjective map $\Psi_{0} : O(n,\br) \rightarrow G(p,n)$ by $\Psi(\varphi) = \varphi(W_{0})$.  The map $\Psi_{0}$ is continuous since $\Psi_{0} = q_{0} \circ \pi_{0}$ and $\pi_{0}$ is a homeomorphism.
\newline

\noindent If H is the subgroup of $O(n,\br)$ that fixes $W_{0}$, then $\Psi_{0}$ induces a bijection $\rho_{0} : O(n,\br) / H \rightarrow G(p,n)$ such that $\Psi_{0} = \rho_{0} \circ p$, where $p : O(n,\br) \rightarrow O(n,\br) / H$ is the projection.  If $O(n,\br) / H$ is given the quotient topology determined by p, then p is an open map.  Moreover, $\rho_{0}$ is continuous since $\Psi_{0}$ is continuous.  The map $\rho_{0} : O(n,\br) \rightarrow G(p,n)$ is a continuous bijection and hence is a homeomorphism since $O(n,\br) / H$ is compact and G(p,n) is Hausdorff.  This proves 2).
\newline

\noindent  It follows that $\Psi_{0} = \rho_{0} \circ p$ is an open map since p is an open map and $\rho_{0}$ is a homeomorphism. Finally, $q_{0} : V_{0}(p,n) \rightarrow G(p,n)$ is an open map since $\Psi_{0} = q_{0} \circ \pi_{0}$ and $\pi_{0}$ is a homeomorphism.  This proves 1).
\end{proof}

\noindent  The next result will be useful later.

\begin{lemma}  Let O be an open subset of $\br^{n}$, and let $W \in G(p,n)$ be a p-plane such that $W \cap O$ is nonempty.  Then there exists a neighborhood V of W in G(p,n) such that $O \cap W'$ is nonempty for every $W' \in V$.
\end{lemma}

\begin{proof}  Let $\{f_{1}, ... , f_{p} \}$ be an orthonormal basis for W, and let $w = \sum_{i=1}^{p} a_{i} f_{i} \in O \cap W$ be given, where $a_{1}, ... , a_{p}$ are real numbers.  For $\epsilon > 0$ let $U_{\epsilon} = \{f' \in V_{0}(p,n) : d(f,f') < \epsilon \}$. Choose $\epsilon > 0$ such that if $f' = (f_{1}', ... , f_{p}' ) \in U_{\epsilon}$, then $w' = \sum_{i=1}^{p} a_{i} f_{i}' \in O$.  The set $U_{\epsilon}$ is open in $V_{0}(p,n)$, and the set $V_{\epsilon} = q_{0}(U_{\epsilon})$ is open in G(p,n) since $q_{0} : V_{0}(p,n) \rightarrow G(p,n)$ is an open map by Lemma 6.1. Note that $V_{\epsilon}$ contains W.  
\newline

\noindent  If $W' \in V_{\epsilon}$, then $W' = q_{0}(f')$, where $d(f,f') < \epsilon$.  It follows by the choice of $\epsilon$ that $Z' = \sum_{i=1}^{p} a_{i} f_{i}' \in W' \cap O$.
\end{proof}

\section{The space of isometry classes of type (p,q)}

\noindent  Let I(p,q) denote the space of isometry classes of metric, 2-step simply connected nilpotent Lie groups $\{N, \langle , \rangle \}$ of type (p,q) or equivalently, of metric, 2-step nilpotent Lie algebras $\{\fN , \langle , \rangle \}$ of type (p,q).  In this section and the next  we characterize I(p,q), and we show that there exists a dense open subset $U' $ in I(p,q) such that if the isometry class of  $\{N, \langle , \rangle \}$ lies in $U'$, then the primitivity condition of Proposition 5.1 is satisfied for all Z in an open subset  with full Lebesgue measure of $[\fN , \fN]$.
\newline

\noindent We first construct a standard form for a metric, 2-step nilpotent Lie algebra of type (p,q).  Let p,q be positive integers with $p \leq \frac{1}{2} q(q-1) = dim~\fs \fo(q,\br)$ and let W be an element of $G(p,\fs \fo(q,\br))$, the Grassmann manifold of p-dimensional subspaces of $\fs \fo(q,\br)$.  Let $\fN = \br^{q} \oplus W$, direct sum.  Let $\langle , \rangle_{0}$ be the positive definite inner product on $\fs \fo(q,\br)$ given by $\langle X,Y \rangle = trace~XY^{T} = - trace~XY$ for $X,Y \in \fs \fo(q,\br)$.  Let $\langle, \rangle$ be the inner product on $\fN$ such that $\langle , \rangle = \langle , \rangle_{0}$ on W, $\langle , \rangle$ is the canonical inner product $\cdot$ on $\br^{q}$ and $\br^{q}$, W are orthogonal.  Then $\fN$ admits a 2-step nilpotent Lie algebra structure of type (p,q) obtained by setting W in the center of $\fN$ and defining  $\langle [v,w], Z \rangle = Z(v) \cdot w$ for all $v,w \in \br^{q}$ and $Z \in W$.  It is easy to show that $[\fN,\fN] = W$.  We say that $\{\fN, \langle , \rangle \}$ is a metric 2-step nilpotent Lie algebra of type (p,q) in $\mathit{standard~ form}$.   Every metric 2-step nilpotent Lie algebra $\{\fN, \langle , \rangle \}$ of type (p,q) is isometric and isomorphic  to one in standard form. See for example Proposition 2.6 of [E2]
\newline

\begin{proposition}  Let $W_{1}, W_{2}$ be elements of $G(p, \fs \fo(q,\br))$, and let $\fN_{1} = \br^{q} \oplus W_{1}$ and $\fN_{2} = \br^{q} \oplus W_{2}$ be the corresponding metric 2-step nilpotent Lie algebras of type(p,q) in standard form.  Then $\fN_{1}$ is isometric and isomorphic to $\fN_{2} \Leftrightarrow$ there exists $g \in O(q,\br)$ such that $gW_{1}g^{-1} = W_{2}$.
\end{proposition}

\noindent  As an immediate consequence we obtain

\begin{corollary}  The space I(p,q) of isometry classes of metric, simply connected, 2-step nilpotent Lie groups of type (p,q) is bijectively equivalent to $G(p,\fs \fo(q,\br)) / O(q,\br)$, where $O(q,\br)$ acts on  $G(p,\fs \fo(q,\br))$ by conjugation.
\end{corollary}

\begin{proof}  We prove Proposition 7.1.  We suppose first that there exists an isometry and Lie algebra isomorphism $\varphi : \fN_{1} \rightarrow \fN_{2}$.  Let $[,]_{1}$ and $[,]_{2}$ denote the Lie algebra structures  on $\fN_{1}$ and $\fN_{2}$ respectively.  Note that $\varphi ([\fN_{1},\fN_{1}]_{1}) = [\fN_{2}, \fN_{2}]_{2}$, or equivalently $\varphi(W_{1}) = W_{2}$, since $\varphi$ is a Lie algebra isomorphism.  Hence $\varphi(\br^{q}) = \br^{q}$ since $\varphi$ is an isometry, and $\br^{q}$ is the orthogonal complement of both $W_{1}$ and $W_{2}$.  Let $g \in O(q,\br)$ denote the restriction of $\varphi$ to $\br^{q}$.
\newline

\noindent  We assert that $gW_{1}g^{-1} = W_{2}$.  Let $x,y \in \br^{q}$ and $Z \in W_{2}$ be given.  From the definitions we obtain $(\varphi^{-1}Z) x\cdot y = \langle [x,y]_{1} , \varphi^{-1}Z \rangle = \langle \varphi([x,y]_{1}), Z \rangle = \langle [\varphi x, \varphi y]_{2}, Z \rangle = Z(\varphi x) \cdot \varphi y = Z(gx) \cdot gy = (g^{-1}Zg) x \cdot y$.  Hence $\varphi^{-1}Z = g^{-1} Z g$ for all $Z \in W_{2}$ since $x,y \in \br^{q}$ were arbitrary.  It follows that $W_{1} = \{\varphi^{-1} Z : Z \in W_{2} \} = g^{-1}W_{2}g$.
\newline

\noindent  Conversely, suppose that $gW_{1} g^{-1} = W_{2}$ for some element g of $O(q,\br)$.  Define $\varphi : \fN_{1} \rightarrow \fN_{2}$ by $\varphi = g$ on $\br^{q}$ and $\varphi(Z) = gZg^{-1}$ for all $Z \in W_{1}$.  We show that $\varphi$ is an isometry and a Lie algebra automorphism.  One sees immediately that $\varphi$ is a linear isometry since conjugation by g is an isometry of $\fs \fo(q,\br)$ with respect to the canonical inner product $\langle , \rangle_{0}$.
\newline

\noindent  We show that $\varphi : \fN_{1} \rightarrow \fN_{2}$ is a Lie algebra automorphism.  It suffices to show that $\varphi([x,y]_{1}) = [\varphi x, \varphi y]_{2}$ for all $x,y \in \br^{q}$.  Let $x,y \in \br^{q}$ and $Z \in W_{2}$ be given.  Then $\langle [\varphi x, \varphi y]_{2} , Z \rangle = Z(\varphi x) \cdot \varphi y = Z(gx) \cdot gy = (g^{-1}Zg) x \cdot y$.  Finally,$ \langle \varphi([x,y]_{1} , Z \rangle = \langle g [x,y]_{1} g^{-1} , Z \rangle = \langle [x,y]_{1}, g^{-1}Zg \rangle = (g^{-1}Zg)x \cdot y \rangle = \langle [\varphi x, \varphi y]_{2} , Z \rangle$.  It follows that  $\varphi([x,y]_{1}) = [\varphi x, \varphi y]_{2}$ for all $x,y \in \br^{q}$ since $Z \in W_{2}$ was arbitrary. 
\end{proof}

\section{metric 2-step nilpotent Lie groups where the maximal primitivity condition holds almost everywhere }

\noindent  A set A in a vector space V over a field K is $\mathit{Zariski~ closed}$ if there exist finitely many polynomials $f_{i} : V \rightarrow K, 1 \leq i \leq m$, such that $A = \{v \in V : f_{i}(v) = 0~for~1 \leq i \leq m \}$.  A set A in V is $\mathit{Zariski~ open}$ if $V - A$ is Zariski closed.  If $K = \br$ or $ \bc$, then V has a natural vector space topology $\fT$ and any Zariski closed subset A of V is closed in $\fT$ and contains no $\fT$-open subset.  In this case Zariski closed subsets have Lebesgue measure zero in V since the sets where the polynomials $\{f_{i} \}$ vanish have dimension lower than that of V.  Hence, if K $= \br$ or $\bc$, then any Zariski open subset A of V is open and dense in V with respect to $\fT$ and has full Lebesgue measure.
\newline

\noindent  Let $q \geq 2$ be given.  For an integer $m \geq 2$ let $O_{m}$ be the set of elements $Z \in \fs \fo(q,\br)$ such that the eigenvalues of Z are distinct and the ratio of any two distinct nonzero eigenvalues is not equal to $m$ or $-m$.  \newline

\begin{proposition}  The set $O_{m}$ is a nonempty Zariski open subset of $\fs \fo(q,\br)$ for every integer $m \geq 2$.  Moreover, $O_{m}$ is invariant under conjugation by elements of $O(q,\br)$ for every $m \geq 2$.
\end{proposition}

\noindent  The $O(q,\br)$ invariance assertion is obvious by inspection.  To avoid becoming bogged down in the proof of the remaining assertion we postpone its proof and include it in an appendix (See Proposition 10.13 and the remarks that follow).  We now come to the main result of this section

\begin{proposition}  Let $q \geq 2$ be an integer, and let $P = [\frac{1}{2} q]$.  Let p be a positive integer with $p \leq \frac{1}{2}q(q-1) = dim~\fs \fo(q,\br)$.  Let $O = \bigcap_{m=2}^{\infty} O_{m}$.  Let $U = \{W \in G(p, \fs \fo(q,\br)  : W \cap O~is~nonempty \}$. For $W \in G(p, \fs \fo(q,\br))$ let $N_{W}$ denote the simply connected 2-step nilpotent Lie group with left invariant metric whose metric Lie algebra is $\fN = \br^{q} \oplus W$.   Then 

	1)  O is an open subset of $\fs \fo(q,\br)$ in the vector space topology invariant under conjugation by elements of $O(q,\br)$.  The set $\fs \fo(q,\br) - O$ has Lebesgue measure zero.
	
	2)  U is a dense open subset of $G(p,\fs \fo(q,\br))$ invariant under conjugation by elements of $O(q,\br)$.
	
	3)  If $W \in U$, then $W \cap O$ is an open set in W in the subspace topology such that $W - (W \cap O)$ has Lebesgue measure zero.
	
	4)  For $W \in U$ and $Z \in W \cap O$ the geodesic $\gamma_{Z}$ in $N_{W}$ has P distinct primitive conjugate values.
\end{proposition}

\begin{proof} 1) The set  $O = \bigcap_{m=2}^{\infty} O_{m}$ is invariant under conjugation by elements of $O(q,\br)$ since each set $O_{m}$ has this property by Proposition 8.1.  By definition O is the set of Z in $\fs \fo(q,\br)$ such that the eigenvalues of Z are distinct and the ratio of any two distinct nonzero eigenvalues is not an integer of absolute value $\geq 2$.  It is evident that O is open in the vector space topology of $\fs \fo(q,\br)$ since the characteristic polynomial of Z, and hence also the eigenvalues of Z, depend continuously on Z.  By the discussion above and Proposition 8.1 the set $\fs \fo(q,\br) - O_{m}$ has Lebesgue measure zero for every $m \geq 2$.  Hence $\fs \fo(q,\br) - O = \bigcup_{m=2}^{\infty} (\fs \fo(q,\br) - O_{m})$ has Lebesgue measure zero.

	2)  Since O is open in the vector space topology of $\fs \fo(q,\br))$ it follows by Lemma 6.2 that U is open in $G(p,\fs \fo(q,\br))$.  The set U is invariant under conjugation by elements of $O(q,\br)$ since O has this property by 1). 
\newline  

\noindent For an integer $m \geq 2$ let $U_{m} = \{W \in G(p,\fs \fo(q,\br)) : O_{m} \cap W~\rm{is~nonempty} \}$.   We show that $U_{m}$ is a dense open subset of  $G(p,\fs \fo(q,\br))$ for all integers $m \geq 2$.  Then we show that  $U = \bigcap_{m=2}^{\infty} U_{m}$ to conclude that U is dense in $G(p,\fs \fo(q,\br))$.
\newline 

\noindent  Note that  each set $U_{m}$ is open in $G(p,\fs \fo(q,\br))$ by Lemma 6.2 since $O_{m}$ is open in $\fs \fo(q,\br)$.  Next we show that $U_{m}$ is dense in $G(p, \fs \fo(q,\br))$ for every integer $m \geq 2$.  Fix an integer $m \geq 2$ and let $W \in G(p,\fs \fo(q,\br))$  be given.  Let $\fB = \{ Z_{1}, ... , Z_{p}\}$ be a basis of W.  By Proposition 8.1 the set $O_{m}$ is dense and open in $\fs \fo(q,\br)$  in the vector space topology.   For each integer i with $1 \leq i \leq p$ let $\{Z_{i}^{(k)} \}$ be a sequence in $O_{m}$ that converges to $Z_{i}$ as $k \rightarrow \infty$.  For sufficiently large k the set $\fB_{k} = \{Z_{1}^{(k)}, ..., Z_{p}^{(k)} \}$ is linearly independent, and without loss of generality we may assume that this is always the case.  Let $W_{k} = \br$-span$\{Z_{1}^{(k)}, ... , Z_{p}^{(k)} \} \in G(p,\fs \fo(q,\br))$.  Clearly $W_{k} \rightarrow W$ in $G(p,\fs \fo(q,\br))$ as $k \rightarrow \infty$ by the continuity of the map $q : V(p,\fs \fo(q,\br)) \rightarrow G(p,\fs \fo(q,\br))$ (see section 6).  On the other hand $W_{k} \in U_{m}$ for every k since $W_{k} \cap O_{m} \supset \fB_{k}$.  Hence $U_{m}$ is dense in $G(p,\fs \fo(q,\br))$.
\newline
	
\noindent To show that U is dense in $G(p,\fs \fo(q,\br))$ it suffices by the Baire category theorem to show that $U = \bigcap_{m=2}^{\infty} U_{m}$.
\newline
	
\noindent  If $W \in U$, then $O \cap W$ is nonempty by the definition of U and hence $O_{m} \cap W$ is nonempty for all $m \geq 2$ by the definition of O.  It follows that $W \in \bigcap_{m=2}^{\infty} U_{m}$, which proves that $U \subseteq \bigcap_{m=2}^{\infty} U_{m}$.  Conversely, let $W \in \bigcap_{m=2}^{\infty} U_{m}$.  Then $O_{m} \cap W$ is nonempty for all $m \geq 2$.  By Proposition 8.1 and Proposition 10.4 $O_{m} \cap W$ is a nonempty subset of W that is open in the Zariski topology of W.  In particular,  $O_{m} \cap W$ is dense and open in the vector space topology of W for all $m \geq 2$.  It follows from the Baire category theorem that $O \cap W = \bigcap_{m=2}^{\infty} O_{m} \cap W$ is dense in W.  In particular, $O \cap W$ is nonempty, which shows that $W \in U$.  We have proved that $\bigcap_{m=2}^{\infty} U_{m} \subseteq U$.
\newline

\noindent 3)  Let $W \in U$ be given.  The set $W \cap O$ is open in the subspace topology of W since O is open in $\fs \fo(q,\br)$.  Moreover, $\bigcap_{m=2}^{\infty} O_{m} \cap W = O \cap W$ is nonempty by the definition of U.  By Proposition 8.1 and Proposition 10.4 the set $O_{m} \cap W$ is nonempty and open in the Zariski topology of W for all $m \geq 2$.  In particular $W - (O_{m} \cap W)$ has Lebesgue measure zero in W for all $m \geq 2$.  It follows that $W - (O \cap W) = \bigcup_{m=2}^{\infty} W - (O_{m} \cap W)$ has Lebesgue measure zero in W.
\newline

\noindent 4).  This follows from 2) of Proposition 5.1, the definition of $O_{m}$ and the fact that $O = \bigcap_{m=2}^{\infty} O_{m}$.
\end{proof}

\begin{corollary}  Let p,q be positive integers with $q \geq 2$ and $p \leq \frac{1}{2}q(q-1)$.  Then there exists a dense open subset $U'$ of I(p,q) with the following property :  Let $\{\fN, \langle , \rangle \}$ be a metric, 2-step nilpotent Lie algebra whose isometry class lies in $U'$, and let $\{N , \langle , \rangle \}$ denote the corresponding simply connected, 2-step nilpotent Lie group with left invariant metric $\langle , \rangle$.  Then there exists a dense open subset O of full Lebesgue measure in $[\fN , \fN]$ such that the geodesic $\gamma_{Z}$ has $P = [\frac{1}{2}q]$ distinct primitive conjugate values for each nonzero element Z of O. 
\end{corollary}

\begin{proof}  Let $\pi : G(p,\fs \fo(q,\br)) \rightarrow G(p,\fs \fo(q,\br)) / O(q,\br) = I(p,q)$ denote the projection map, where $O(q,\br)$ acts on $G(p,\fs \fo(q,\br))$ by conjugation.  Let U be the dense open subset in $G(p,\fs \fo(q,\br))$ of Proposition 8.2.  Then $U' = \pi(U)$ is dense and open in I(p,q) since $\pi$ is an open, continuous, surjective map.   By Proposition 8.2 the set $U'$ satisfies the assertions of Corollary 8.3. 
\end{proof}

\section{Weak conjugacy in $G(p, \fs \fo(q,\br))$}

\noindent  Two elements $W_{1}, W_{2}$ in $G(p, \fs \fo(q,\br))$ are $\mathit{conjugate}$ if $gW_{1}g^{-1} = W_{2}$ for some $g \in O(q,\br)$.  We have seen in section 7 that $W_{1}, W_{2}$ are conjugate $\Leftrightarrow$ there exists a Lie algebra isomorphism $\varphi : \br^{q} \oplus W_{1} \rightarrow \br^{q} \oplus W_{2}$ that is also a linear isometry.
\newline

\noindent  We say that two elements $W_{1}, W_{2}$ in $G(p, \fs \fo(q,\br))$ are $\mathit{weakly~ conjugate}$ if there exists a linear isomorphism $\varphi : W_{1} \rightarrow W_{2}$ such that for every Z in $W_{1}, \varphi(Z)$ is conjugate to Z by an element g(Z) in $O(q,\br)$ that may depend on Z.  Note that such a map $\varphi$ must also be a linear isometry since $|\varphi(Z)|^{2} = - trace~ \varphi(Z) ^{2} = - trace~(g(Z)Z g(Z)^{-1})^{2} = - trace~g(Z) Z^{2} g(Z)^{-1} = - trace~Z^{2} = |Z|^{2}$ for all $Z \in W_{1}$.  
\newline

\noindent  It is natural to ask if any two weakly conjugate  elements $W_{1}, W_{2}$ in $G(p, \fs \fo(q,\br))$ are actually conjugate.  This is known to be false if $p = 2$.  See Theorem 2.2 of [GW].  It is unknown if this question has an affirmative answer for sufficiently large p ; as p increases the condition of weak conjugacy becomes more stringent.The next result shows that the weak conjugacy question has geometric relevance.
\newline

\begin{proposition}   Let an integer $q \geq 2$ and an integer p with $1 \leq p \leq \frac{1}{2}q(q-1)$ be given.  Let U be the dense open subset of $G(p, \fs \fo(q,\br))$ defined in Proposition 8.2.  Let elements $W_{1},W_{2}$ in $G(p, \fs \fo(q,\br))$ with $W_{1} \in U$ be given.  Let $N_{1},N_{2}$ be the simply connected, 2-step nilpotent Lie groups with left invariant metrics whose metric Lie algebras are $\fN_{1} = \br^{q} \oplus W_{1}, \fN_{2} = \br^{q} \oplus W_{2}$ respectively.  Let $\varphi : W_{1} \rightarrow W_{2}$ be a vector space isomorphism such that the geodesics  $\gamma_{Z}$ and $\gamma_{\varphi(Z)}$ have the same conjugate locus in $N_{1}$ and $N_{2}$ respectively.  Then $\varphi : W_{1} \rightarrow W_{2}$ is a weak conjugacy.
\end{proposition}

\noindent $\mathbf{Remark}$  In the statement of the Proposition we do not assume that $\gamma_{Z}$ and $\gamma_{\varphi(Z)}$ have the same multiplicities for a fixed conjugate value of both geodesics.  The fact that these multiplicities are the same then follows immediately from  the conclusion of this Proposition and Corollary 4.5.
\newline

\begin{proof}  Let Z be an arbitrary element of $W_{1}$.  Let $t > 0$ be a primitive conjugate value for the geodesic $\gamma_{Z}$ in $N_{1}$.  We show first that $t > 0$ is also a primitive conjugate value for the geodesic $\gamma_{\varphi(Z)}$ in $N_{2}$.  If this were not the case, then we could write $t = mt_{0}$, where $m \geq 2$ is an integer and $t_{0}$ is a conjugate value of $\gamma_{\varphi(Z)}$ in $N_{2}$.  By hypothesis $t_{0}$ is also a conjugate value for $\gamma_{Z}$ in $N_{1}$, which contradicts the hypothesis that $t > 0$ is a primitive conjugate value for the geodesic $\gamma_{Z}$.
\newline  

\noindent  Now let O be the subset of those elements Z in $\fs \fo(q,\br)$ such that the eigenvalues of j(Z) are all distinct and the ratio of any two distinct eigenvalues is not an integer with absolute value $\geq 2$.  If $Z \in O \cap W_{1}$, then by 4) of Proposition 8.2 the geodesic $\gamma_{Z}$ in $N_{1}$ has P distinct primitive conjugate values $\{t_{1}, ... , t_{P} \}$.  By the hypothesis on $\varphi$,  the remarks above and 1) of Proposition 5.1 the geodesic $\gamma_{\varphi(Z)}$ in $N_{2}$ has the same primitive values $\{t_{1}, ... , t_{P} \}$.  It follows from Proposition 3) of Proposition 5.1 and Corollary 5.2 that Z and $\varphi(Z)$ have the same eigenvalues, all of multiplicity one.  In particular, Z and $\varphi(Z)$ have the same characteristic  polynomial for all $Z \in O \cap W_{1}$.
\newline

\noindent  Now let $Z \in W_{1}$ be given.  Since $W_{1} \in U$ it follows from 3) of Proposition 8.2 that $O \cap W_{1}$ is open and dense in $W_{1}$ in the vector space topology.  Hence there exists a sequence $\{Z_{k} \}$ in $O \cap W_{1}$ such that $Z_{k} \rightarrow Z$ as $k \rightarrow \infty$.  By the observations above $Z_{k}$ and $\varphi(Z_{k})$ have the same characteristic polynomial for all k.  Hence by continuity Z and $\varphi(Z)$ have the same characteristic polynomial.  It follows from the finite dimensional spectral theorem that Z and $\varphi(Z)$ are conjugate by some element g(Z) of $O(q,\br)$.
\end{proof}

\section{Appendix : Polynomial maps and the Zariski topology}

\noindent $\mathbf{Polynomial~ maps~ and~ Zariski~ open~ sets}$

\noindent Let V be a vector space over a field K, not necessarily finite dimensional, and let F(V) denote the K-algebra of functions from V to K.  Let K[V] denote the K-subalgebra of F(V) generated by V$^{*}$.  The elements of K[V] are the $\mathit{polynomial~functions}$ on V.
\newline

\noindent  A set $\Sigma \subset V$ is $\mathit{Zariski~closed}$ in V if there exist finitely many elements $g_{1}, ... , g_{m}$ of K[V] such that $\Sigma = \{v \in V : g_{i}(v) = 0~for~1 \leq i \leq m \}$.  A set $O \subset V$ is $\mathit{Zariski~open}$ in V if $V - O$ is Zariski closed in V.  If $K = \br$ or $\bc$, then V has a standard vector space topology, and Zariski open subsets are open and dense in V with respect to this topology. 
\newline

\noindent $\mathbf{Polynomial~maps~between~vector~spaces}$

\noindent  Let V,W be vector spaces over K, and let $g : V \rightarrow W$ be a function.  The function g is said to be a $\mathit{polynomial}$ function if $T \circ g \in K[V]$ for all $T \in K[W]$.
\newline

\begin{proposition}  A function $g : V \rightarrow W$ is a polynomial function $\Leftrightarrow T \circ g \in K[V]$ for all $T \in W^{*}$. 
\end{proposition}

\begin{proof}  If  $g : V \rightarrow W$ is a polynomial function, then $T \circ g \in K[V]$ for all $T \in W^{*}$ since $W^{*} \subset K[W]$.  Conversely, suppose that $T \circ g \in K[V]$ for all $T \in W^{*}$.  Let $S = \{f \in K[W] : f \circ g \in K[V] \}$.   By hypothesis $S \supset W^{*}$, and it is easy to check that S is a K-subalgebra of K[W].  Hence S = K[W], which completes the proof. 
\end{proof}

\begin{corollary}  Let $g : V \rightarrow W$ be a linear map.  Then g is a polynomial map.
\end{corollary}
 
\begin{proof}  If $T \in W^{*}$, then $T \circ g \in V^{*} \subset K[V]$.  Now apply Proposition 10.1.
\end{proof}

\begin{proposition}  (Chain Rule)  Let V,W,Z be vector spaces over K.  Let $g : V \rightarrow W$ and $h : W \rightarrow Z$ be polynomial maps.  Then $h \circ g : V \rightarrow Z$ is a polynomial map.
\end{proposition}

\begin{proof}  If $T \in K[Z]$, then $T \circ h \in K[W]$ since h is a polynomial map, and it follows that $T \circ (h \circ g) = (T \circ h) \circ g \in K[V]$ since g is a polynomial map.
\end{proof}

\begin{proposition}  Let V be a vector space over K, and let W be a subspace of V.  Let $i : W \rightarrow V$ be the inclusion map.  Then i is continuous with respect to the Zariski topologies on V and W.
\end{proposition}

\begin{proof}  Let A be a  subset of V that is closed in the Zariski topology of V.  It suffices to show that $i^{-1}(A) = A \cap W$ is closed in the Zariski topology of W.  Let $p_{1}, ... , p_{m} \in K[V]$ be polynomials such that $A = \{v \in V : p_{k}(v) = 0~for~1 \leq k \leq m \}$.  If $q_{k}$ is the restriction of $p_{k}$ to W, then $q_{k} = p_{k} \circ i$ for $1 \leq k \leq m$.  It follows from Corollary 10.2 and Proposition 10.3 that $q_{k} \in K[W]$ for $1 \leq k \leq m$.  It follows from the definitions that $A \cap W = \{w \in W : q_{k}(v) = 0~for~1 \leq k \leq m \}$, which completes the proof.
\end{proof} 

\noindent $\mathbf{Restrictions~ of~ polynomial~ maps}$

\begin{proposition}  Let V,W be vector spaces over K, and let U be a subspace of V.  Let $g : V \rightarrow W$ be a polynomial map, and let $g_{U} : U \rightarrow W$ denote the restriction of g to U.  Then $g_{U}$ is a polynomial map.
\end{proposition}

\begin{proof}  Let $i : U \rightarrow V$ denote the inclusion map.  Then i is a linear map and hence a polynomial map by Corollary 10.2.  Hence $g_{U} = g \circ i$ is a polynomial map by the chain rule.
\end{proof}

\begin{proposition}  Let V,W be vector spaces over K, and let W have a basis. Let $g : V \rightarrow W$ be a polynomial map, and let U be a subspace of W such that $g(V) \subset U$.  Then $g : V \rightarrow U$ is a polynomial map.
\end{proposition}

\begin{proof}  It is known that there exists a basis $\fB$ of W that contains a basis $\fB'$ of U.  See for example the discussion in chapter IX of [J].  Let $U' = K$- span $\{\fB - \fB' \}$.  Then $U \oplus U' = W$.  If $T \in U^{*}$, then one can extend T to an element $\tilde{T}$ of $W^{*}$ ; for example, define $\tilde{T} = T$ on U and $\tilde{T} = 0$ on $U'$.  
\newline

\noindent  Now let $T \in U^{*}$ be given and choose an element $\tilde{T} \in W^{*}$ such that $\tilde{T} = T$ on U.  Then $T \circ g = \tilde{T} \circ g \in K[V]$ since $g : V \rightarrow W$ is a polynomial map.  Hence $g : V \rightarrow U$ is a polynomial map by Proposition 10.1.
\end{proof}

\begin{proposition}  Let V,W be vector spaces over K, and let W have a basis.  Let K[V,W] denote the K-vector space of polynomial maps from V to W.  If W is a K-algebra, then K[V,W] is a K-algebra.
\end{proposition}

\begin{proof} Let W be a K-algebra, and let $g,h : V \rightarrow W$ be polynomial maps.  Define $g \cdot h : V \rightarrow W$ by $(g \cdot h)(v) = g(v) \cdot h(v)$.  It remains to show that $g \cdot h \in K[V,W]$. 
\newline

\begin{lemma}  Let V,W be vector spaces over K, and let W have a basis $\fB = \{w_{\alpha} : \alpha \in A \}$.  Let $\fB^{*} = \{w_{\alpha}^{*} : \alpha \in A \}$ be the dual basis of $W^{*}$.  Let $g : V \rightarrow W$ be any function and define $g_{\alpha} = w_{\alpha}^{*} \circ g : V \rightarrow K$ for every $\alpha \in A$.  Then

	1)  $g = \sum_{\alpha \in A} g_{\alpha}~w_{\alpha}$, where the sum is finite when evaluated on each v of V.
	
	2)  g is a polynomial map $\Leftrightarrow g_{\alpha} : V \rightarrow K$ is a polynomial map for each $\alpha \in A$.
\end{lemma}

\begin{proof} 1)  If $v \in V$, then $g(v) = \sum_{\alpha \in A'} a_{\alpha} w_{\alpha}$ for some finite subset set $A' $ of $A$. Then $g_{\beta}(v) = a_{\beta}$ for every $\beta \in A'$ and $g_{\beta}(v) = 0$ for $\beta \in A - A'$. This proves 1). 
\newline

\noindent We prove 2).  Clearly, if $g : V \rightarrow W$ is a polynomial map, then $g_{\alpha} = w_{\alpha}^{*} \circ g : V \rightarrow K$ is a polynomial map for each $\alpha \in A$ since $w_{\alpha}^{*} \in W^{*}$.  Conversely, suppose that $g_{\alpha} = w_{\alpha}^{*} \circ g : V \rightarrow K$ is a polynomial map for every $\alpha \in A$.  Let $w^{*} \in W^{*}$ be given and write $w^{*} = \sum_{\alpha \in A'} a_{\alpha} w_{\alpha}^{*}$, where $A'$ is a finite subset of A, and $a_{\alpha} \in K$ for every $\alpha \in A$. Then $w^{*} \circ g = \sum_{\alpha \in A'} a_{\alpha} (w_{\alpha}^{*} \circ g) = \sum_{\alpha \in A'} a_{\alpha} g_{\alpha} \in K[V]$ since each $g_{\alpha}$ lies in K[V].  Now apply Proposition 10.1 to complete the  proof of 2).
\end{proof}

\noindent We now complete the proof of the Proposition. Let $\fB = \{w_{\alpha} : \alpha \in A\}$ be a basis for W, and let $\fB^{*} = \{w_{\alpha}^{*} : \alpha \in A \}$ be the corresponding dual basis for $W^{*}$.  For $\alpha, \beta \in A$ write $w_{\alpha} \cdot w_{\beta} = \sum_{\gamma \in A'} A_{\alpha \beta}^{\gamma} w_{\gamma}$, where $A'$ is a finite subset of A and $A_{\alpha \beta}^{\gamma} \in K$ for all $\alpha, \beta, \gamma$.   Let $g_{\alpha} = w_{\alpha}^{*} \circ g$ and $h_{\beta} = w_{\beta}^{*} \circ h$ for each $\alpha, \beta \in A$.  Then $g_{\alpha} \in K[V]$ and $h _{\beta} \in KV]$ for all $\alpha, \beta \in A$.
\newline

\noindent Given $v \in V$, Lemma 10.8 shows that $(g \cdot h)(v) = g(v) \cdot h(v) = (\sum_{\alpha \in A} g_{\alpha}(v) w_{\alpha}) \cdot (\sum_{\beta \in A} h_{\beta}(v) w_{\beta}) = \sum_{\alpha, \beta \in A} g_{\alpha}(v) h_{\beta}(v) w_{\alpha} \cdot w_{\beta} = \sum_{\alpha, \beta \in A} g_{\alpha}(v) h_{\beta}(v) (\sum_{\gamma \in A'} A_{\alpha \beta}^{\gamma} w_{\gamma}) = \sum_{\gamma \in A'}(\sum_{\alpha, \beta \in A} g_{\alpha}(v) h_{\beta}(v) A_{\alpha \beta}^{\gamma}) w_{\gamma}$.  Hence $w_{\gamma}^{*} \circ (g \cdot h) = \sum_{\alpha, \beta \in A} g_{\alpha} h_{\beta} A_{\alpha \beta}^{\gamma} \in K[V]$ for all $\gamma \in A$.  It follows as in the proof of 2) of Lemma 10.8 that $w^{*} \circ (g \cdot h) \in K[V]$ for every $w^{*} \in W^{*}$.  Hence $g \cdot h \in K[V]$ by Proposition 10.1.
\end{proof}

\noindent $\mathbf{Applications}$

\begin{proposition}  Let n be a positive integer, and let M(n,K) denote the K-vector space of n x n matrices with entries in K.  Let V be a subspace of M(n,K).  Define $\varphi : V \rightarrow K[t]$ by $\varphi(Z) = det(tI - Z)$ for all $Z \in V$.  Then  $\varphi : V \rightarrow K[t]$ is a polynomial map.
\end{proposition}

\begin{proof}  Define $\tilde{\varphi} : M(n,K) \rightarrow K[t]$ by $\tilde{\varphi}(Z) = det(tI - Z)$ for all $Z \in M(n,K)$.  Clearly $\tilde{\varphi} = \varphi$ on V.  By Proposition 10.5 it suffices to prove that $\tilde{\varphi} : M(n,K) \rightarrow K[t]$ is a polynomial map.
\newline

\noindent  Let $\{E_{ij} : 1 \leq i,j \leq n \}$ be the natural basis for M(n,K), and let $\{E_{ij}^{*} : 1 \leq i,j \leq n \}$ be the dual basis for M(n,K)$^{*}$.  The K-algebra K[t] has a natural basis $\fB = \{1,t,t^{2}, ... , t^{m}, ... \}$.  Let $\fB^* = \{T_{0}, T_{1}, T_{2}, ... , T_{m}, ...\}$ be the dual basis for K[t]$^{*}$.  For a positive integer r let $\tilde{\varphi_{r}} = T_{r} \circ \tilde{\varphi}$.  By Proposition 10.1 and the proof of 2) of Lemma 10.8 it suffices to show $\tilde{\varphi_{r}} : M(n,K) \rightarrow K$ is a polynomial map for every integer $r \geq 0$.
\newline

\noindent Let Z $\in$ M(n,K) be given.  Then $\tilde{\varphi}(Z) = det (tI - Z) = \sum_{k=0}^{n} t^{k} P_{k}(Z_{ij})$, where $Z_{ij} = E_{ij}^{*}(Z)$ for $1 \leq i,j \leq n$ and $P_{k}$ is a polynomial in the $n^{2}$ entries $\{Z_{ij} \}$.  It follows that $\tilde{\varphi_{r}}(Z) = P_{r}(Z_{ij}) = (P_{r}(E_{ij}^{*}))$(Z) for every integer $r \geq 0$.  Hence $\tilde{\varphi_{r}} = (P_{r}(E_{ij}^{*}))$, a polynomial in $\{E_{ij}^{*} \}$, and we conclude that $\tilde{\varphi_{r}} : M(n,K) \rightarrow K$ is a polynomial map for every integer $r \geq 0$.
\end{proof}

\begin{proposition}  Let $q = 2k$ and let $V = \fs \fo(q,\br) \subset M(q,\br)$.  Define $\varphi : V \rightarrow \br[t]$ by $\varphi(Z) = det (tI - Z)$.  For an integer m with $|m| \geq 2$ define $\varphi_{m} : V \rightarrow \br[t]$ and $\psi_{m} : V \rightarrow \br[t]$ by $\varphi_{m}(Z) = det (tI - mZ)$ and $\psi_{m}(Z) = \varphi(Z) \cdot \varphi_{m}(Z)$.  Then $\psi_{m}$ is a polynomial map for every m.
\end{proposition}

\begin{proof}  This follows immediately from Propositions 10.7 and 10.9.
\end{proof}

\begin{proposition} Let $q = 2k + 1$ and let $V = \fs \fo(q,\br) \subset M(q,\br)$.   Let $W = (t^{2})$, the ideal in $\br[t]$ generated by the polynomial $t^{2}$.  For an integer m with $|m| \geq 2$ define $\psi_{m} : V \rightarrow W$ by $\psi_{m}(Z) = det (tI - Z) \cdot det (tI - mZ)$.  Then $\psi_{m}$ is a polynomial map.
\end{proposition}

\begin{proof}  Fix $Z \in V$.  Note that both $\varphi(t) = det (tI - Z)$ and $\varphi_{m}(t) = det (tI - mZ), m \geq 2$, have $t=0$ as a root since q is odd and both Z and mZ are skew symmetric.  Hence $\psi_{m}(V) \subset W$. From Propositions 10.7 and 10.9  it follows that $\psi_{m} : V \rightarrow \br[t]$ is a polynomial map.  By Proposition 10.6 it follows that  $\psi_{m} : V \rightarrow W$ is a polynomial map for every integer m with $|m| \geq 2$.
\end{proof}

\begin{corollary} Let $q = 2k + 1$ and let $V = \fs \fo(q,\br) \subset M(q,\br)$.  For an integer m with $|m| \geq 2$ define $\tilde{\psi_{m}} : V \rightarrow \br[t]$ by $\tilde{\psi_{m}}(Z) = det (tI - Z) \cdot det (tI - mZ) / t^{2}$.  Then $\psi_{m}$ is a polynomial map.
\end{corollary}

\begin{proof}  Let $W = (t^{2})$, the ideal in $\br[t]$ generated by the polynomial $t^{2}$.  Let $\psi_{m} : V \rightarrow W$ be the polynomial map defined in Proposition 10.11.  Let $\pi : W \rightarrow \br[t]$ be the linear map given by $\pi(f) = f / t^{2}$.  Then $\tilde{\psi_{m}} = \pi \circ \psi_{m}$ for all $m \geq 2$.  The map $\pi : W \rightarrow \br[t]$ is polynomial since it is linear, and it now follows from the chain rule that $\tilde{\psi_{m}}$ is a polynomial map for all integers m with $|m| \geq 2$.
\end{proof}

\begin{proposition}  Let m,q be integers with $q \geq 2$ and $|m| \geq 2$.  Let $A_{m}$ be the set of those elements in $\fs \fo(q,\br)$ such that 

	1)  The eigenvalues of Z are distinct in $\bc$.
	
	2)  The ratio of any two distinct nonzero eigenvalues of Z is not equal to m.
		
\noindent  Then $A_{m}$ is a nonempty Zariski open subset of $\fs \fo(q,\br)$ 
\end{proposition}

\noindent $\mathbf{Remarks}$  

	1)  The eigenvalues of $Z \in \fs \fo(q,\br)$  lie in $i \br$, so the ratio of any two nonzero eigenvalues lies in $\br$.
	
	2)  If $O_{m} = A_{m} \cap A_{-m}$, then $O_{m}$ is nonempty and Zariski open in $\fs \fo(q,\br)$ by 2) of the Proposition.  The set $O_{m}$ appears in the statement of Proposition 9.1.
\newline

\noindent We shall consider separately the cases that q is even or odd, which correspond to Proposition 10.10 and Corollary 10.12 respectively.  Before we can address either case we need the following 

 \begin{proposition}  Let K be a field of characteristic zero, and let $\overline{K}$ be an algebraically closed field that contains K.  For a positive integer k let  $K_{k}[t] = K-span \{1,t,^{2}, ... , t^{k} \} \subset K[t]$. Let $O_{k} = \{f(t) \in K_{k}[t] : f(t)~ \rm{has~ k~ distinct~ roots~ in} ~\overline{K} \}$.  Then $O_{k}$  is a nonempty Zariski open subset of $K_{k}[t]$.  
\end{proposition}

\begin{proof}  Clearly $O = O_{k}$ is nonempty, so it suffices to prove that O is Zariski open. Fix an integer $k \geq 1$ and let $\pi_{0} : K_{k}[t] \rightarrow K$ be defined by $\pi_{0}(a_{0}t^{k} + a_{1}t^{k-1} + ... + a_{k-1}t + a_{k}) = a_{0}$.  Let $O_{1}= \{f(t) \in K_{k}[t] : \pi_{0}(f) \neq 0 \}$.  The map $\pi_{0}$ is linear hence polynomial (cf. Corollary 10.2). The Zariski open subset $O_{1}$ consists of those polynomials in $K_{k}[t]$ of degree k, or equivalently, those polynomials in $K_{k}[t]$ with k roots in $\overline{K}$.
\newline

\noindent For f(t) $\in O_{1}$ let D(f) $\in\overline{K}$ be given by $D(f) = \prod_{i<j} (\alpha_{i} - \alpha_{j})^{2}$, where $\{\alpha_{1},...,\alpha_{k} \}$ is the set of roots of f(t) in $\overline{K}$.  Clearly f(t) has k distinct roots in $\overline{K} \Leftrightarrow D(f) \neq 0$. 
	
\begin{lemma}  There exists a polynomial map g : $K_{k}[t] \rightarrow K$ and a positive integer $N'$ such that $D(f) = (1/\pi_{0}(f))^{N'} g(f)$ if f $\in O_{1}$.
\end{lemma}

\noindent We postpone the proof of the Lemma to complete the proof of the Proposition.  We assert that $O = O_{1} \cap O_{2}$, where O and O$_{1}$ have been defined above and $O_{2} = \{f(t) \in K_{k}[t] : g(f) \neq 0 \}$.  It will then follow that O is Zariski open in $K_{k}[t]$.
\newline
	
\noindent Let f(t) $\in$ O be given.  Then $ \pi_{0}(f) \neq 0$ since f has k roots in $\overline{K}$, and hence f $\in O_{1}$.  Moreover, $0 \neq D(f) = (1 / \pi_{0}(f)) ^{N} g(f)$ since f has k distinct roots in $\overline{K}$.  Hence $f \in O_{2}$, which shows that $O \subseteq O_{1} \cap O_{2}$.  Conversely, let f $\in O_{1} \cap O_{2}$ be given. Then $\pi_{0}(f) \neq 0$ since f(t) $\in O_{1}$, and we conclude that f(t) has k roots in $\overline{K}$.  In addition, $D(f) = (1/ \pi_{0}(f))^{N} g(f) \neq 0$ since f(t) $\in O_{2}$.  Hence f has k distinct roots in $\overline{K}$, which shows that $O_{1} \cap O_{2} \subseteq O$.
\newline
	
\noindent We now begin the proof of the Lemma.  Let $f(t) = a_{0}t^{k} + a_{1}t^{k-1} + ... + a_{k-1}t + a_{k}$ be an element of $K_{k}[t]$ with $a_{0} \neq 0$, and let $\{\alpha_{1},...,\alpha_{k} \}$ be the roots of f(t).  If $\overline{f}(t) = (1/a_{0}) f(t) = t^{k} + (a_{1}/a_{0})t^{k-1} + ... + (a_{k-1}/a_{0}) t + (a_{k}/a_{0})$, then $\overline{f}(t)$ has the same roots as f(t).  Since the coefficient of $t^{k}$ for $\overline{f}(t)$ is 1 we obtain $\overline{f}(t) = \prod_{i=1}^{k} (t - \alpha_{i}) = \sum_{i=1}^{k} (-1)^{i} s_{i}(\alpha_{1},...,\alpha_{k}) t^{i}$, where $s_{i}(t_{1},...,t_{k})$ is the i$^{th}$ elementary symmetric function of the variables $t_{1},...,t_{k}$  .  Comparing the two expressions for $\overline{f}(t)$ yields
	
		(1)  $(-1)^{i} a_{k-i}/a_{0} = s_{i}(\alpha_{1},...,\alpha_{k})$ for $1 \leq i \leq k$
		
\noindent If $D(t_{1},...,t_{k})$ is the discriminant polynomial given by $D(t_{1},...,t_{k}) = \prod_{i<j} (t_{i} -t_{j})^{2}$, then the polynomial $D(t_{1},...,t_{k})$ is symmetric in the variables $t_{1},...,t_{k}$ and has coefficients in $\bz \subset K$.   By the theory of symmetric polynomial functions it is known (cf. Theorem 1 of chapter IV, p.62 of [B])  that there exists a polynomial $h(t_{1},...,t_{k})$ in the variables $t_{1},...,t_{k}$ with coefficients in $\bz$ such that $D(t_{1},...,t_{k}) = h(s_{1},...,s_{k})$.  From (1) we obtain
	
		(2)  $D(f) = h(\frac{-a_{k-1}}{a_{0}},\frac{a_{k-2}}{a_{0}},...,\frac{(-1)^{k-1}a_{1}}{a_{0}})$
		
\noindent since $D(f) = D(\alpha_{1},...,\alpha_{k}) = h(s_{1}(\alpha_{1},...,\alpha_{k}),..., s_{k}(\alpha_{1},...,\alpha_{k}))$.  By inspection there exists a positive integer $N'$, depending only on h, such that 

	(3) $a_{0}^{N'}~ h(\frac{-a_{k-1}}{a_{0}},\frac{a_{k-2}}{a_{0}},...,\frac{(-1)^{k-1}a_{1}}{a_{0}}) = h'(a_{0},a_{1},...,a_{k})$
	
\noindent where h$'$ is a polynomial in the variables $t_{0},t_{1},...,t_{k}$ with coefficients in $\bz$.  Let      $g = h' \circ \pi$, where $\pi : K_{k}[t] \rightarrow K^{k+1}$ is the isomorphism given by $\pi(a_{0}t^{k} + a_{1}t^{k-1} + ... + a_{k-1}t + a_{k}) = (a_{0},a_{1},...,a_{k})$.  Then $(1/a_{0})^{N'} g(f) = (1/a_{0})^{N'} h'(a_{0},a_{1},...,a_{k}) = D(f)$ by (2) and (3).  Note that g is a polynomial function such that $g(K_{k}[t]) =  h'(K^{k+1}) \subset K$ since h$'$ has coefficients in $\bz \subset K$.  This completes the proof of the Lemma.
\end{proof}

\noindent  We now prove Proposition 10.13, and we consider first the case that q is even.  For each positive integer k let $\br_{k}[t] = \br-span~\{1,t,t^{2}, ... ,t^{k} \} \subset \br[t]$. For each integer m with $|m| \geq 2$ define a map $\psi_{m} : \fs \fo(q,\br) \rightarrow \br_{2q}[t]$ by $\psi_{m}(Z) = det (tI - Z) \cdot det (tI - mZ)$ for all $Z \in \fs \fo(q,\br)$.  The map $\psi_{m}$ is a polynomial map by Proposition 10.10.  If $Z \in A_{m}$, then the eigenvalues of Z are all nonzero since they are distinct and q is even. Hence the roots of $\psi_{m}(Z)$ are all nonzero.
\newline

\noindent  We use the notation of the proof of Proposition 10.14.  Let $O_{1} = \{f(t) \in \br_{2q}[t] : \pi_{0}(f) \neq 0 \}$.  By the proof of Proposition 10.14 there exists a polynomial map $ g : \br_{2q}[t] \rightarrow \br$ and a positive integer $N'$ such that $D(f) = \frac{1}{\pi_{0}(f)^{N'}} g(f)$ for all $f \in O_{1}$.  By inspection $\psi_{m}(\fs \fo(q,\br) \subset O_{1}$ and $(D \circ \psi_{m})(Z) = \frac{1}{\pi_{0}(\psi_{m}(Z))^{N'} }(g \circ \psi_{m})(Z)$ for all $Z \in \fs \fo(q,\br)$.
\newline

\noindent  Let O be the elements in $\fs \fo(q,\br)$ whose eigenvalues are all distinct.  As noted above, if $Z \in O$, then the eigenvalues of Z are all nonzero since q is even.  The set O is a Zariski open subset of $\fs \fo(q,\br)$ by Proposition 10.14, and the set $O_{1}$ is clearly a Zariski open subset of $\fs \fo(q,\br)$ by its definition.  Since $g \circ \psi_{m} : \fs \fo(q,\br) \rightarrow \br$ is a polynomial map by the chain rule it suffices to show that $A_{m} = \{Z \in O \cap O_{1} : g \circ \psi_{m}(Z) \neq 0 \}$.
\newline

\noindent For $Z \in \fs \fo(q,\br)$ the eigenvalues of mZ are m times the eigenvalues of Z.  If $Z \in O \cap O_{1}$, then the ratio of two distinct eigenvalues of Z equals m $\Leftrightarrow$ the polynomials  $det (tI -Z)$ and $det(tI - mZ)$ have a common root $\Leftrightarrow \psi_{m}(t)$ has a root of multiplicity at least two $\Leftrightarrow (g \circ \psi_{m})(Z) = 0$.  This concludes the proof in the case that q is even.
\newline

\noindent  We consider the case that q is odd.  The proof of the Proposition in this case is basically the same as in the case that q is even, but the proof must be modified since $\psi_{m}(t) = det (tI - Z) \cdot det (tI - mZ)$ is divisible by $t^{2}$.  We use Corollary 10.12 to overcome this technical problem. 
\newline

\noindent Let O and $O_{1}$ be the same Zariski open sets defined above in the case that q is even.  For an integer m with $|m| \geq 2$ we define $\tilde{\psi_{m}} : \fs \fo(q,\br) \rightarrow \br[t]$ by $\tilde{\psi_{m}}(t) = det (tI - Z) det (tI - mZ) / t^{2}$.  The map $\tilde{\psi_{m}}$ is a polynomial map by Corollary 10.12.  If $Z \in O \cap O_{1}$, then the roots of $det (tI - Z)$ are distinct and the roots of $det (tI - mZ)$ are distinct.  Hence zero is a root of multiplicity one for both polynomials.  It follows that  the roots of $\tilde{\psi_{m}}(Z)$ are all nonzero for all $Z \in O \cap O_{1}$.
\newline

\noindent Arguing as above, for $Z \in O \cap O_{1}$ the ratio of two nonzero eigenvalues of Z equals m $\Leftrightarrow \tilde{\psi_{m}}(Z)$ has a root of multiplicity at least two $\Leftrightarrow (g \circ \tilde{\psi_{m}})(Z) = 0$, where $g : \br_{2q}[t] \rightarrow \br$ is the polynomial in the proof of the case where q is even.  This completes the proof of the Proposition.
\newline

\noindent Patrick Eberlein

\noindent Department of Mathematics

\noindent University of North Carolina

\noindent Chapel Hill, NC 27599

\noindent USA

\noindent pbe@email.unc.edu
\newline
 
\noindent $\mathbf{References}$
\newline

\noindent [B]  N. Bourbaki, Elements of Mathematics, Algebra II, Chapters 4-7, Springer, 1990 (English Edition).
\newline

\noindent [CE]  J. Cheeger and D. Ebin, Comparison Theorems in Riemannian Geometry, North-Holland, 1975.
\newline

\noindent [E1] P. Eberlein, "Geometry of 2-step nilpotent groups with a left invariant metric", Ann. Scient. Ecole Normale Sup. (1994), 611-660.
\newline

\noindent [E2] -------, "Riemannian submersions and lattices in 2-step nilpotent Lie groups", Comm. in Analysis and Geom. 11(3), (2003), 441-488.
\newline

\noindent [GW]  C. Gordon and E. Wilson, "Continuous families of isospectral Riemannian manifolds which are not locally isometric", Jour. Diff. Geom. 47 (3), (1997), 504-529.
\newline

\noindent [J] N. Jacobson, Lectures in Abstract Algebra, vol. II, Van Nostrand, New York, 1953.
\newline

\noindent [JLP] C. Jang, T. Lee and K. Park, "Conjugate loci of 2-step nilpotent Lie groups" satisfying $J_{z}^{2} = \langle Sz,z \rangle A$", J. Korean Math. Soc. 45(6),(2008), 1705-1723.
\newline

\noindent [M] J. Milnor, Morse Theory, Annals of Math. Studies 51, Princeton University Press, Princeton, 1963.

\noindent [MS] J. Milnor and J. Stasheff, Characteristic Classes, Annals of Math Studies 76, Princeton University Press, Princeton, 1974.
\newline

\noindent [R]  M.S. Raghunathan,Discrete Subgroups of Lie groups, Springer, New York, 1972.

\end{document}